\documentclass[12pt,english]{amsart}
\usepackage[a4paper]{geometry}
\usepackage{amssymb,amsmath}

\numberwithin{equation}{section}

\newtheorem*{Thmuncount}{Theorem}
\newtheorem{theorem}{Theorem}[section]

\newtheorem*{rem}{Remark}
\newenvironment{unremark}{\begin{rem}\rm}{\end{rem}}

\newtheorem{proposition}[theorem]{Proposition}
\newtheorem{fact}[theorem]{Fact}
\newtheorem{lemma}[theorem]{Lemma}
\newtheorem{corollary}[theorem]{Corollary}
\newtheorem{Definition}[theorem]{Definition}

\newtheorem{Remark}[theorem]{Remark}
\newenvironment{remark}{\begin{Remark}\rm}{\end{Remark}}
\newtheorem{Notation}[theorem]{Notation}

\newtheorem{RHproblem}[theorem]{RH problem}

\newtheorem{Example}[theorem]{Example}

\theoremstyle{definition}

\renewcommand{\P}{\mathbb{P}}

\newcommand{\B}{\mathbb{B}}
\newcommand{\C}{\mathbb{C}}
\newcommand{\D}{\mathbb D}

\newcommand{\R}{\mathbb{R}}

\newcommand{\T}{\mathbb{T}}

\renewcommand{\P}{\mathbb{P}}

\newcommand{\CC}{\mathcal C}

\usepackage[bookmarksopen, naturalnames]{hyperref}

\usepackage[english]{babel}

\begin{document}

\title{Small Bergman-Orlicz and Hardy-Orlicz spaces, and their composition operators}

\author{S. Charpentier}
\address{St\'ephane Charpentier, Institut de Math\'ematiques, UMR 7373, Aix-Marseille Universit\'e, 39 rue F. Joliot Curie, 13453 Marseille Cedex 13, FRANCE}
\email{stephane.charpentier.1@univ-amu.fr}

\begin{abstract}
We show that the weighted Bergman-Orlicz space $A_{\alpha}^{\psi}$ coincides with some weighted Banach space of holomorphic functions if and only if the Orlicz function $\psi$ satisfies the so-called $\Delta^{2}$--condition. In addition we prove that this condition characterizes those $A_{\alpha}^{\psi}$ on which every composition operator is bounded or order bounded into the Orlicz space $L_{\alpha}^{\psi}$. This provides us with estimates of the norm and the essential norm of composition operators on such spaces. We also prove that when $\psi$ satisfies the $\Delta^{2}$--condition, a composition operator is compact on $A_{\alpha}^{\psi}$ if and only if it is order bounded into the so-called Morse-Transue space $M_{\alpha}^{\psi}$. Our results stand in the unit ball of $\mathbb{C}^{N}$.
\end{abstract}



\maketitle
\def\leftmark{\MakeUppercase{}}

\def\rightmark{\MakeUppercase{St\'ephane Charpentier}}

\section{Introduction and first definitions}

Hardy-Orlicz and Bergman-Orlicz spaces are natural generalizations of the classical Hardy $H^p$ and Bergman spaces $A^p$. They have been rather well studied when the defining Orlicz function grows slowly, \emph{i.e.} when they are larger than $H^1$ and, roughly speaking, similar to the Nevanlinna class $\mathcal{N}$, which is one instance of such a space. Nevertheless, there are not so many papers dealing with small Hardy-Orlicz and Bergman-Orlicz spaces. One of the interests in looking at these spaces is that, due to the wide range of Orlicz functions, they provide an as well wide and refined scale of natural Banach spaces of holomorphic functions between $A^1$ and $H^{\infty}$. And by many aspects $H^{\infty}$ radically differs from Hardy and Bergman spaces. This can be seen for instance from the point of view of operator theory and more specifically from that of composition operators. We recall here that a composition operator $C_{\phi}$, with symbol $\phi$, is defined as $C_{\phi}(f)=f\circ \phi$, where $f$ lies in some Banach space of holomorphic functions and $\phi$ is an appropriate holomorphic map. If $\B_N=\left\{ z=\left(z_{1},\ldots,z_{N}\right)\in\mathbb{C}^{N},\,\left|z\right|:=\left(\left|z_{1}\right|^2+\ldots+\left|z_{N}\right|^2\right)^{1/2}<1\right\}$ is the unit ball of $\mathbb{C}^{N}$, every composition operator is trivially bounded on the space $H^{\infty}(\B_N)$ of bounded holomorphic functions on $\B_N$ and, when $N=1$, the Littlewood Subordination Principle says that this is also true on every Hardy and Bergman space \cite{shapiro_composition_1993}. Yet things get more complicated whenever $N\geq 2$ since even the simple symbol $\varphi(z)=(N^{N/2}z_1z_2\ldots z_N,0')$, $0'$ being the null $N-1$-tuple, does not induce a bounded composition operator on $H^2(\B_N)$ (see \cite{cowen_composition_1995} for other examples). Besides, whatever the dimension, while the compactness of $C_{\phi}$ on $H^p(\B_N)$, $1\leq p < \infty$, does not depend on $p$ and can occur when $\phi(\B_N)$ touches the boundary of $\B_N$ \cite{cowen_composition_1995,shapiro_compact_1973}, it is compact on $H^{\infty}(\B_N)$ if and only if $|\phi(z)|\leq r <1$ for every $z\in \B_N$ \cite{schwartz_composition_1969}. These observations recently motivated some authors to study composition operators on the whole scale of Hardy-Orlicz and Bergman-Orlicz spaces. This study turned out to be rich, \cite{lefevre_compact_2009,lefevre_composition_2010,lefevre_revisited_2012,lefevre_compact_2013} on the unit disc $\D$ and \cite{charpentier_operateurs_2010,charpentier_composition_2011,charpentier_compact_2013,charpentier_composition_2013} on the unit ball $\B_N$, $N\geq 1$. In the present paper we will mainly focus on \emph{small} Hardy-Orlicz and Bergman-Orlicz spaces, namely those which are \emph{close} to $H^{\infty}(\B_N)$. A surprising characterization of the latter ones will be obtained and used to refine some previously known results. We hope that it will highlight how properties of composition operators are related to the \emph{structure} of these spaces.

Let us give the definitions of Hardy-Orlicz and weighted Bergman-Orlicz spaces. Given an \emph{Orlicz function} $\psi$, \emph{i.e.} a strictly convex function $\psi:[0,+\infty)\rightarrow[0,+\infty)$
which vanishes at $0$, is continuous at $0$ and satisfies ${\displaystyle \frac{\psi\left(x\right)}{x}\xrightarrow[x\rightarrow+\infty]{}+\infty}$,
and given a probability space $\left(\Omega,\mathbb{P}\right)$, the \emph{Orlicz
space} $L^{\psi}\left(\Omega,\mathbb{P}\right)$ associated with $\psi$
on $\left(\Omega,\mathbb{P}\right)$ is defined as the set of all
(equivalence classes of) measurable functions $f$ on $\Omega$ for
which there exists some $C>0$, such that $\int_{\Omega}\psi\left(\frac{\left|f\right|}{C}\right)d\mathbb{P}$
is finite. It is also usual to define the Morse-Transue space $M^{\psi}\left(\Omega,\mathbb{P}\right)$
as the subspace of $L^{\psi}\left(\Omega,\mathbb{P}\right)$ for which
$\int_{\Omega}\psi\left(\frac{\left|f\right|}{C}\right)d\mathbb{P}$
is finite for any constant $C>0$. $L^{\psi}\left(\Omega,\mathbb{P}\right)$
and $M^{\psi}\left(\Omega,\mathbb{P}\right)$, endowed with the \emph{Luxemburg
norm} $\left\Vert \cdot\right\Vert _{\psi}$ given by
\[
\left\Vert f\right\Vert _{\psi}=\inf\left\{ C>0,\,\int_{\Omega}\psi\left(\frac{\left|f\right|}{C}\right)d\mathbb{P}\leq 1\right\} ,
\]
are Banach spaces. Then for $\alpha>-1$, we introduce the \emph{weighted
Bergman-Orlicz space} $A_{\alpha}^{\psi}\left(\mathbb{B}_{N}\right)$
and the \emph{little weighted Bergman-Orlicz space} $AM_{\alpha}^{\psi}\left(\mathbb{B}_{N}\right)$
as the spaces of those functions in $H\left(\mathbb{B}_{N}\right)$, the algebra of all holomorphic functions on $\B_N$,
which also belongs to $L^{\psi}\left(\mathbb{B}_{N},v_{\alpha}\right)$
and $M^{\psi}\left(\mathbb{B}_{N},v_{\alpha}\right)$ respectively, where $dv_{\alpha}=c_{\alpha}(1-|z|^2)^{\alpha}dv$ is the normalized weighted Lebesgue measure on the ball.
For any Orlicz function $\psi$, $A_{\alpha}^{\psi}\left(\mathbb{B}_{N}\right)$ and $AM_{\alpha}^{\psi}\left(\mathbb{B}_{N}\right)$, $\alpha >-1$, endowed with the Luxemburg norm $\left\Vert\cdot \right\Vert _{A_{\alpha}^{\psi}}$ inherited from $L^{\psi}(\B_N,v_{\alpha})$, are Banach spaces.

Similarly generalizing the usual Hardy spaces we define the Hardy-Orlicz
space $H^{\psi}\left(\mathbb{B}_{N}\right)$ as follows:
\[
H^{\psi}\left(\mathbb{B}_{N}\right):=\left\{ f\in H\left(\mathbb{B}_{N}\right),\,\left\Vert f\right\Vert _{H^{\psi}}:=\sup_{0<r<1}\left\Vert f_{r}\right\Vert _{\psi}<\infty\right\} ,
\]
where $\left\Vert \cdot\right\Vert _{\psi}$ is the Luxembourg norm
on the Orlicz space $L^{\psi}\left(\mathbb{S}_{N},\sigma_{N}\right)$
and $f_{r}$ is given by $f_{r}(z)=f(rz)$. Here $\sigma_N$ stands for the normalized rotation-invariant positive Borel measure on the unit sphere $\mathbb{S}_N=\partial \B_N$. For any Orlicz function
$\psi$, $H^{\psi}\left(\mathbb{B}_{N}\right)$ can be identified with a closed subspace of $L^{\psi}\left(\mathbb{S}_{N},\sigma_{N}\right)$. We then define the \emph{little Hardy-Orlicz space} $HM^{\psi}\left(\mathbb{B}_{N}\right)$ as $H^{\psi}\left(\mathbb{B}_{N}\right)\cap M^{\psi}\left(\mathbb{S}_{N},\sigma_{N}\right)$. Endowed with $\left\Vert \cdot \right\Vert _{H^{\psi}}$, $H^{\psi}\left(\mathbb{B}_{N}\right)$ and $HM^{\psi}\left(\mathbb{B}_{N}\right)$ are Banach spaces \cite{charpentier_composition_2011}.

When $\psi(x)=x^{p}$, $1\leq p<\infty$, $A_{\alpha}^{\psi}\left(\mathbb{B}_{N}\right)$
and $H^{\psi}\left(\mathbb{B}_{N}\right)$ are the classical weighted
Bergman space $A_{\alpha}^{p}\left(\mathbb{B}_{N}\right)$ and Hardy
space $H^{p}\left(\mathbb{B}_{N}\right)$, respectively. Moreover we have the following \cite{rao_theory_1991}
\[
H^{\infty}=\bigcap_{\psi\mbox{ Orlicz}}A_{\alpha}^{\psi}\left(\mathbb{B}_{N}\right)=\bigcap_{\psi\mbox{ Orlicz}}H^{\psi}\left(\mathbb{B}_{N}\right),
\]
while trivially $A_{\alpha}^{1}\left(\mathbb{B}_{N}\right)=\cup_{\psi\mbox{ Orlicz}}A_{\alpha}^{\psi}\left(\mathbb{B}_{N}\right)$
and $H^{1}\left(\mathbb{B}_{N}\right)=\cup_{\psi\mbox{ Orlicz}}H^{\psi}\left(\mathbb{B}_{N}\right)$. We refer
to \cite{charpentier_composition_2013,lefevre_compact_2013} for an introduction
to weighted Bergman-Orlicz spaces and to \cite{charpentier_composition_2011,lefevre_compact_2009}
for a discussion about Hardy-Orlicz spaces (on the disc and on the ball); to learn more about the
general theory of Orlicz spaces, we refer to \cite{krasnoselskii_convex_1961,rao_theory_1991}.

As $A_{\alpha}^{\psi}(\B_N)$ and $H^{\psi}(\B_N)$ are continuously embedded into $H(\B_N)$, some sharp inclusions
$$A_{\alpha}^{\psi}\left(\mathbb{B}_{N}\right)\subseteq H_{v}^{\infty}\left(\mathbb{B}_{N}\right)\mbox{ and }H^{\psi}\left(\mathbb{B}_{N}\right)\subseteq H_{w}^{\infty}\left(\mathbb{B}_{N}\right)$$
hold for some \emph{typical radial weights} $v$ and $w$, where $H_{v}^{\infty}\left(\mathbb{B}_{N}\right)$ is the \emph{weighted Banach space} - or \emph{growth space}- associated with $v$:
\begin{equation}\label{growth-space}
H_{v}^{\infty}:=\left\{ f\in H\left(\mathbb{B}_{N}\right),\,\left\Vert f\right\Vert _{v}:=\sup_{z\in\mathbb{B}_{N}}\left|f(z)\right|v(z)<\infty\right\}.
\end{equation}
We recall that a \emph{weight} $v$ on $\mathbb{B}_{N}$
is a continuous (strictly) positive function on $\mathbb{B}_{N}$,
which is \emph{radial} if $v(z)=v(|z|)$ for all $z\in\mathbb{B}_{N}$, and \emph{typical} if $v(z)\rightarrow 0$ as $|z|\rightarrow 1$.

\medskip{}

Our first result will say that when the Orlicz function $\psi$ grows \emph{fast}, then $A_{\alpha}^{\psi}(\B_N)$ coincides with $H_v^{\infty}(\B_N)$ for $\displaystyle{v(z)=\left[\psi^{-1}\left(\frac{1}{1-|z|}\right)\right]^{-1}}$. For instance, if $\psi(x)=e^x-1$, then $A_{\alpha}^{\psi}(\B_N)=H_v^{\infty}(\B_N)$ with $v(z)=\left[\log(\frac{e}{1-|z|})\right]^{-1}$, which is the typical growth for the functions in the Bloch space. This result echoes Kaptano\u{g}lu's observation that some weighted Bloch spaces (not the true ones) coincide with some other weighted Banach spaces \cite{kaptanoglu_carleson_2007}. An immediate consequence is that weighted and unweighted \emph{small} Bergman-Orlicz are the same. We should mention that weighted Banach spaces and their composition operators have been far more studied than Bergman-Orlicz spaces, \cite{abakumov_holomorphic_nodate,aron_spectra_2004,bierstedt_associated_1998,bonet_essential_1999,bonet_note_2003,hyvarinen_essential_2012,lusky_weighted_1995,lusky_isomorphism_2006,lusky_weighted_2010,montes-rodriguez_weighted_2000} and the references therein. We will mention from time to time in the forthcoming sections that similar results were obtained independently for $H_v^{\infty}(\D)$ and $A_{\alpha}^{\psi}(\D)$ ($\D:= \B_1$), and turn out to be identical for $\psi$ and $v$ as above.

We will also derive easily from the previous that every composition operator is bounded on $A_{\alpha}^{\psi}(\B_N)$ whenever $\psi$ grows \emph{fast}, a result which was already stated in \cite{charpentier_composition_2013} and where the machinery of Carleson measures was heavily used. Actually, the proof of our first result will point out that, more than bounded, every composition operator is \emph{order bounded} into $L^{\psi}(\B_N,v_{\alpha})$ ($\psi$ growing \emph{fast}). We recall that an operator $T$ from a Banach space $X$ to a Banach lattice $Y$ is order bounded into some subspace $Z$ of $Y$ if there exists $y\in Z$ positive such that $\sup_{\left\Vert x\right\Vert \leq 1}\left|Tx\right|\leq y$. The notion of order boundedness is often related to that of $p$-summing, nuclear or Dunford-Pettis operators \cite{diestel_absolutely_1995}, as illustrated by the result of A. Shields, L. Wallen and J. Williams, asserting that an operator from a normed space to some $L^p$ space, which is order bounded into $L^p$, is $p$-summing. From the point of view of composition operators, the notion was investigated in several papers \cite{hunziker_composition_1991,jaoua_order_1999,jarchow_functional_1995,ueki_order_2012}. On $H^{\infty}(\B_N)$ every bounded operator is trivially order bounded into $L^{\infty}(\B_N)$. For the classical Hardy and Bergman spaces $H^p(\D)$, it is shown in \cite{shapiro_compact_1973} that $C_{\phi}$ is order bounded into $L^p(\partial \D)$ if and only if it $p$-summing, when $2\leq p < \infty$ (\emph{i.e.} Hilbert-Schmidt when $p=2$). The same holds in $\B_N$ and for the Bergman spaces. The situation becomes radically different when one studies the order boundedness of composition operators on Hardy-Orlicz and Bergman-Orlicz spaces which are neither reflexive nor separable (\emph{i.e.} when $H^{\psi}(\B_N)$ and $A_{\alpha}^{\psi}(\B_N)$ differs from $HM^{\psi}(\B_N)$ and $AM_{\alpha}^{\psi}(\B_N)$ respectively). Indeed it appears, in such cases, that the order boundedness of $C_{\phi}$ acting on $H^{\psi}(\B_N)$ (resp. $A_{\alpha}^{\psi}(\B_N)$ into $L^{\psi}(\mathbb{S}_N,\sigma _N)$ (resp. $L^{\psi}(\B_N,v_{\alpha})$) is no longer stronger than its compactness (by the foregoing, this is already clear when $\psi$ grows \emph{fast}, since every composition operator is then order bounded into $L^{\psi}(\B_N,v_{\alpha})$, while the identity fails to be compact). Though $C_{\phi}$ is still compact for any $\psi$ whenever it is order bounded into $M^{\psi}(\mathbb{S}_N,\sigma _N)$ (resp. $M^{\psi}(\B_N,v_{\alpha})$), the latter does not imply any more that it is $p$-summing, for any $p<\infty$ \cite{lefevre_compact_2009}. Note that on $H^{\infty}$, $C_{\phi}$ is compact if and only if it is $p$-summing \cite{lefevre_characterizations_2005}.

In these directions, our two main results state as follows. The first one (Theorem \ref{thm|thm-equiv-OB-B-Delta2}) completes what has been said above.

\begin{Thmuncount}Let $M\geq 1$, $N>1$, $\alpha>-1$,
and $\psi$ be an Orlicz function. The following assertions are equivalent.
\begin{enumerate}
\item $\psi$ satisfies the $\Delta^{2}$--condition (a fast growth condition, see Section \ref{Section1})
\item $A_{\alpha}^{\psi}\left(\mathbb{B}_{M}\right)=H_{v}^{\infty}\left(\mathbb{B}_{M}\right)$, where $\displaystyle{v(z)=\left[\psi^{-1}\left(\frac{1}{1-|z|}\right)\right]^{-1}}$;
\item For every $\phi:\mathbb{B}_{M}\rightarrow\mathbb{B}_{M}$ holomorphic,
$C_{\phi}$ acting on $A_{\alpha}^{\psi}\left(\mathbb{B}_{M}\right)$
is order-bounded into $L^{\psi}(\B_M,v_{\alpha})$;
\item For every $\phi:\mathbb{B}_{N}\rightarrow\mathbb{B}_{N}$ holomorphic,
$C_{\phi}$ is bounded on $A_{\alpha}^{\psi}\left(\mathbb{B}_{N}\right)$;
\item $C_{\varphi}$ is bounded on $A_{\alpha}^{\psi}\left(\B_N\right)$, where $\varphi(z)=(N^{N/2}z_1z_2\ldots z_N,0')$.
\end{enumerate}
\end{Thmuncount}

The second one (Theorem \ref{main-thm-compacntess-Deltaup2}) is a generalization of \cite[Theorem 3.24]{lefevre_compact_2009} to $N\geq 2$ and to the Bergman-Orlicz setting.

\begin{Thmuncount}
Let $\alpha > -1$, let $\psi$ be an Orlicz function satisfying
the $\Delta^{2}$--condition and let $\phi:\mathbb{B}_{N}\rightarrow\mathbb{B}_{N}$
be holomorphic. The following assertions are equivalent.
\begin{enumerate}
\item $C_{\phi}$ is compact on $H^{\psi}(\B_N)$ (resp. $A_{\alpha}^{\psi}\left(\mathbb{B}_{N}\right)$);
\item $C_{\phi}$ acting on $H^{\psi}(\B_N)$ (resp. $A_{\alpha}^{\psi}\left(\mathbb{B}_{N}\right)$)
is order bounded into $M^{\psi}(\mathbb{S}_N,\sigma _N)$ (resp. $M^{\psi}(\B_N,v_{\alpha})$);
\item $C_{\phi}$ is weakly compact on $H^{\psi}(\B_N)$ (resp. $A_{\alpha}^{\psi}\left(\mathbb{B}_{N}\right)$).
\end{enumerate}
\end{Thmuncount}

\medskip{}

By many aspects the behavior of a composition operator $C_{\phi}$ is related to how frequently and sharply $\phi(\B_N)$ touches the boundary of $\B_N$. In particular, it is known that if $\phi(\B_N)$ is contained in some \emph{Kor\'anyi approach regions} (\emph{i.e.} \emph{Stolz domains} for $N=1$), then $C_{\phi}$ is bounded or compact on $H^p(\B_N)$ and $A_{\alpha}^p(\B_N)$, depending on the opening of the Kor\'anyi region \cite{cowen_composition_1995,shapiro_composition_1993}. In fact the condition on the opening of the Kor\'anyi region $K$ which ensures that $\phi(\B_N)\subset K$ implies $C_{\phi}$ compact also ensures that $\phi(\B_N)\subset K$ implies $C_{\phi}$ order bounded into $L^{\psi}(\mathbb{S}_N,\sigma _N)$. In \cite{charpentier_compact_2013} it was proven that $\phi(\B_N)$ included in any such region does no longer automatically imply the compactness of $C_{\phi}$ on $H^{\psi}(\B_N)$ or $A_{\alpha}^{\psi}(\B_N)$. In the last section, we will geometrically illustrate the idea developed in the previous paragraph by showing that if $\phi(\B_N)$ is included in some Kor\'anyi region, then $C_{\phi}$ is order bounded into $L^{\psi}(\B_N,v_{\alpha})$ (resp. $L^{\psi}(\mathbb{S}_N,\sigma _N)$), whatever $\psi$. The faster the Orlicz function will grow, the less restrictive will be the condition on the opening of the region.

We shall mention that we will make use of Carleson measure techniques only for the equivalence between (1) and (3) in our second main theorem (through Corollary \ref{cor|corollaire_Berezin_cont_comp}).

The paper is organized as follows: in the next section we will provide the remaining definitions and show that \emph{small} weighted Bergman-Orlicz spaces coincide with some weighted Banach spaces. In Section 3, we will focus on the boundedness, the compactness, and the order boundedness of composition operators acting on \emph{small} weighted
Bergman-Orlicz spaces and Hardy-Orlicz spaces. The last section will be devoted to some geometric illustration of the general idea of the paper.
\begin{Notation}
Given two points $z,w\in\mathbb{C}^{N}$, the euclidean inner product
of $z$ and $w$ will be denoted by $\left\langle z,w\right\rangle $,
that is $\left\langle z,w\right\rangle =\sum_{i=1}^{N}z_{i}\overline{w_{i}}$;
the notation $\left|\cdot\right|$ will stand for the associated norm,
as well as for the modulus of a complex number.

The notation $\T$ will be used to denote the unit sphere $\mathbb{S}_1=\partial \D$.


We recall that $dv_{\alpha}=c_{\alpha}(1-|z|^2)^{\alpha}dv$ is the normalized weighted Lebesgue measure on the ball and so that $c_{\alpha}=\frac{\Gamma(N+\alpha+1)}{N!\Gamma(\alpha +1)}$, where $\Gamma$ is the Gamma function.

$d\sigma _N$ will stand for the normalized rotation-invariant positive Borel measure on $\mathbb{S}_N$.

We will use the notations $\lesssim$ and $\gtrsim$ for one-sided
estimates up to an absolute constant, and the notation $\approx$
for two-sided estimates up to absolute constants.
\end{Notation}

\section{A menagerie of spaces}

\subsection{Weighted Bergman-Orlicz spaces and Hardy-Orlicz spaces}\label{Section1}

\subsubsection*{Some important classes of Orlicz functions.}

Orlicz-type spaces can be distinguished by the regularity and the
growth at infinity of the defining Orlicz function. We recall that
two \emph{equivalent} Orlicz functions define the same Orlicz space,
where two Orlicz functions $\Psi_{1}$ and $\Psi_{2}$ are said to be equivalent
if there exists some constant $c$ such that
\[
c\Psi_{1}(cx)\leq\Psi_{2}(x)\leq c^{-1}\Psi_{1}(c^{-1}x),
\]
for any $x$ large enough. In the sequel, we will consider Orlicz
functions essentially of four types:
\begin{itemize}
\item $\psi$ satisfies the $\Delta_{2}$--condition (or belongs to the $\Delta_{2}$--class)
if there exists $C>0$ such that $\psi(2x)\leq C\psi(x)$ for any
$x$ large enough.
\item $\psi$ satisfies the $\Delta^{1}$--condition (or belongs to the $\Delta^{1}$--class)
if there exists $C>0$ such that $x\psi(x)\leq\psi(Cx)$ for any $x$
large enough .
\item $\psi$ satisfies the $\Delta^{2}$--condition (or belongs to the $\Delta^{2}-$class)
if there exists $C>0$, such that $\psi\left(x\right)^{2}\leq\psi\left(Cx\right)$
for any $x$ large enough.
\item $\psi$ satisfies the $\nabla_{2}$--condition (or belongs to the $\nabla_{2}$--class)
if its complementary function $\Phi$ belongs to the $\Delta_{2}$--class,
where $\Phi$ is defined by
\[
\Phi(y)=\sup_{x\geq0}\left\{ xy-\psi(x)\right\} ,\, y\geq0.
\]
\end{itemize}
The $\nabla_{2}$--condition is a regularity condition satisfied by
most of the Orlicz functions, and $\Delta_{2}$, $\Delta^{1}$ and
$\Delta^{2}$--conditions are growth conditions. It is easily seen
that the $\Delta^{2}$--condition implies the $\Delta^{1}$--condition
which in its turn implies the $\nabla_{2}$. The $\Delta_{2}$--class is disjoint from both $\Delta^1$ and $\Delta^2$--classes.

Using that Orlicz functions are increasing convex functions and an
easy induction, we get the following characterizations of $\Delta_{2}$,
$\Delta^{1}$ and $\Delta^{2}$--classes \cite{charpentier_operateurs_2010,krasnoselskii_convex_1961}.
\begin{proposition}
\label{prop|prop-Orlicz-growth}Let $\psi$ be an Orlicz function.
\begin{enumerate}
\item $\psi$ belongs to the $\Delta_{2}$--class if and only if for one
$b>1$ (or equivalently all $b>1$) there exists $C>0$ such that
$\psi(bx)\leq C\psi(x)$ for any $x$ large enough.
\item $\psi$ belongs to the $\Delta^{1}$--class if and only if for one
$b>1$ (or equivalently all $b>1$) there exists $C>0$ such that
$x^{b}\psi(x)\leq \psi(Cx)$ for any $x$ large enough.
\item $\psi$ belongs to the $\Delta^{2}$--class if and only if for one
$b>1$ (or equivalently all $b>1$) there exists $C>0$ such that
$\psi(x)^{b}\leq \psi(Cx)$ for any $x$ large enough.
\end{enumerate}
\end{proposition}
\smallskip{}
If $\psi\in\Delta_{2}$, then $x^{q}\lesssim\psi(x)\lesssim x^{p}$
for some $1\leq q\leq p<+\infty$ and $x$ large enough; typical examples
are $x\mapsto ax^{p}\left(\left(\log(1+x)\right)^{b}\right)$, $1<p<+\infty$,
$a>0$ and $b>0$. If $\psi\in\Delta^{1}$, then $\psi(x)\geq\exp\left(a(\log x)^{2}\right)$
for some $a>0$ and $x$ large enough. The $\Delta^{2}$--class consists
of functions $\psi$ satisfying $\psi(x)\geq e^{x^a}$
for some $a>0$ and for any $x$ large enough. Orlicz functions of
the form $x\mapsto\exp\left(a(\log(x+1))^{b}\right)-1$ with $a>0$
and $b\geq2$ are examples of $\Delta^{1}$-functions, which are not in the $\Delta^{2}$--class.
Moreover $x\mapsto e^{ax^{b}}-1$ belongs to the $\Delta^{2}-$class,
whenever $a>0$ and $b\geq1$. Note that  $x\mapsto x^{p}$,
$1<p<\infty$, satisfies the $\nabla_{2}$--condition, and
not the $\Delta^{1}$--one.

\smallskip{}

We do not go into more details and we refer to \cite{charpentier_compact_2013,lefevre_compact_2009}
for information sufficient for our purpose. Actually the classification of Orlicz functions is much more refined; to learn more about
the general theory of Orlicz spaces, see \cite{krasnoselskii_convex_1961} or \cite{rao_theory_1991}.\medskip{}

\medskip{}

As in the classical case, it can be convenient for the presentation
to sometimes look at Hardy-Orlicz spaces as a limit case of weighted
Bergman-Orlicz spaces $A_{\alpha}^{\psi}\left(\mathbb{B}_{N}\right)$
when $\alpha$ tends to $-1$. The following standard lemma gives
some sense to that.
\begin{lemma}
\label{lem|lemme-alpha-tend-moins-1}Let $f:\mathbb{\mathbb{B}_N}\mapsto[0,+\infty[$
be continuous and such that the function $r\mapsto\int_{\mathbb{S}_{N}}f(rz)d\sigma_{N}(z)$, $r\in [0,1)$,
is increasing. Then
\[
\lim_{r\rightarrow1}\int_{\mathbb{S}_{N}}f(r\zeta)d\sigma_{N}(\zeta)=\lim_{\alpha\rightarrow-1}\int_{\mathbb{B}_{N}}f(z)dv_{\alpha}(z),
\]
where the limits may be $\infty$.\end{lemma}
\begin{proof}
Let $f$ be as in the lemma and let us write, for $-1<\alpha<0$ and
$0<\eta<1$,
\begin{multline}\label{multlem}
\int_{\mathbb{B}_{N}}f(z)dv_{\alpha}(z)=2Nc_{\alpha}\int_{0}^{1-\eta}\int_{\mathbb{S}_{N}}f(r\zeta)d\sigma_{N}(\zeta)r^{2N-1}\left(1-r^{2}\right)^{\alpha}dr\\+2Nc_{\alpha}\int_{1-\eta}^{1}\int_{\mathbb{S}_{N}}f(r\zeta)d\sigma_{N}(\zeta)r^{2N-1}\left(1-r^{2}\right)^{\alpha}dr.
\end{multline}
Observe that because $c_{\alpha}\rightarrow0$ as $\alpha\rightarrow-1$, the first term of the right-hand side tends to $0$. Since $F:r\mapsto\int_{\mathbb{S}_{N}}f(rz)d\sigma_{N}(z)$ is increasing, either $F$ is convergent, or it increases to $\infty$ as $r\rightarrow 1$. Assume first that $F(r)\rightarrow \infty$ as $r\rightarrow 1$. Then for any $A>0$ there exists $\eta(A)\in (0,1)$, close enough to $1$, such that
$$\int_{\mathbb{S}_{N}}f(r\zeta)d\sigma_{N}(\zeta) \geq A,\,r\in (1-\eta ,1).$$
Moreover $2Nc_{\alpha}\int_{0}^{1}\left(1-r^{2}\right)^{\alpha}r^{2N-1}dr=1$ and $2Nc_{\alpha}\int_{0}^{1-\eta}\left(1-r^{2}\right)^{\alpha}r^{2N-1}dr\rightarrow 0$ as $\alpha \rightarrow -1$, hence
\begin{equation}\label{eq00}2Nc_{\alpha}\int_{1-\eta}^{1}\left(1-r^{2}\right)^{\alpha}r^{2N-1}dr\rightarrow 1\text{ as }\alpha \rightarrow -1,
\end{equation}
for any $\eta \in (0,1)$.
Thus
$$2Nc_{\alpha}\int_{1-\eta(A)}^{1}\int_{\mathbb{S}_{N}}f(r\zeta)d\sigma_{N}(\zeta)r^{2N-1}\left(1-r^{2}\right)^{\alpha}dr\geq 2ANc_{\alpha}\int_{1-\eta(A)}^{1}\left(1-r^{2}\right)^{\alpha}r^{2N-1}dr,$$
which tends to $A$ as $\alpha \rightarrow -1$. Since the first term of the right-hand side of (\ref{multlem}) is non-negative, we finally get
$$\lim _{\alpha \rightarrow -1}\int_{\mathbb{B}_{N}}f(z)dv_{\alpha}(z) \geq A$$
for any $A$ positive, that is
$$\lim _{\alpha \rightarrow -1}\int_{\mathbb{B}_{N}}f(z)dv_{\alpha}(z)=\lim _{r\rightarrow 1}\int_{\mathbb{S}_{N}}f(rz)d\sigma_{N}(z)=\infty.$$

Let now assume that $\lim _{r\rightarrow 1}F(r)$ exists in $(0,\infty)$. For every $\varepsilon >0$ one can find $\eta(\varepsilon) \in (0,1)$ close enough to $1$ such that
$$\int_{\mathbb{S}_{N}}f(r\zeta)d\sigma_{N}(\zeta)-\lim_{r\rightarrow 1}F(r)\leq \varepsilon,\,r\in (1-\eta,1).$$
Using (\ref{eq00}) we get, for any $\varepsilon >0$,
\begin{multline*}\lim_{\alpha \rightarrow -1}\int_{\mathbb{B}_{N}}f(z)dv_{\alpha}(z) - \lim _{r\rightarrow 1}F(r)\\ = \lim_{\alpha \rightarrow -1}2Nc_{\alpha}\int_{1-\eta}^{1}\left(\int_{\mathbb{S}_{N}}f(r\zeta)-\lim_{r\rightarrow 1}F(r)\right)\left(1-r^{2}\right)^{\alpha}r^{2N-1}dr \leq \varepsilon,
\end{multline*}
as expected.
\end{proof}

In particular, given $g\in H(\mathbb{B}_{N})$, $\psi$ an Orlicz
function and $a>0$, and applying Lemma \ref{lem|lemme-alpha-tend-moins-1}
to the subharmonic function $f=\psi\left(\frac{\left|g\right|}{a}\right)$
we get ${\displaystyle \left\Vert g\right\Vert _{H^{\psi}}}=\lim_{\alpha\rightarrow-1}\left\Vert g\right\Vert _{A_{\alpha}^{\psi}}$.

For now on we will use the unique notation $A_{\alpha}^{\psi}\left(\mathbb{B}_{N}\right)$ to denote
the weighted Bergman-Orlicz space when $\alpha>-1$, and the Hardy-Orlicz
space $H^{\psi}\left(\mathbb{B}_{N}\right)$ when $\alpha=-1$. More
abusively, $L_{\alpha}^{\psi}$ (resp. $M_{\alpha}^{\psi}$), $\alpha\geq-1$,
will often stand for $L^{\psi}\left(\mathbb{B}_{N},v_{\alpha}\right)$
(resp. $M^{\psi}\left(\mathbb{B}_{N},v_{\alpha}\right)$) when $\alpha>-1$
and for $L^{\psi}\left(\mathbb{S}_{N},\sigma_{N}\right)$ (resp. $M^{\psi}\left(\mathbb{S}_{N},\sigma_{N}\right)$)
when $\alpha=-1$, $v_{\alpha}$ being $\sigma_{N}$ when $\alpha=-1$.

\medskip{}

The following result specifies the topological and dual properties
of Hardy-Orlicz and Bergman-Orlicz spaces, pointing out that if we
intend to generalize the classical Hardy or Bergman spaces and provide with a refined scale of spaces up to $H^{\infty}$, then we must consider
Banach spaces with less nice properties.
\begin{theorem}
[\cite{charpentier_composition_2011,charpentier_composition_2013}]\label{thm|weak-star-topo-H-O}Let $\alpha \geq -1$ and let $\psi$ be an Orlicz function.
\begin{enumerate}
\item $AM_{\alpha}^{\psi}\left(\mathbb{B}_{N}\right)$ is the closure of $H^{\infty}$ in $L_{\alpha}^{\psi}$; in particular, the set of all polynomials is dense in $AM^{\psi}\left(\mathbb{B}_{N}\right)$.
\item On the unit ball of $AM_{\alpha}^{\psi}\left(\mathbb{B}_{N}\right)$ (or equivalently on \emph{every} ball), the induced weak-star topology coincides with that of uniform convergence on compacta of $\mathbb{B}_{N}$.
\item $A_{\alpha}^{\psi}\left(\mathbb{B}_{N}\right)$ is weak-star closed in $L_{\alpha}^{\psi}$.
\item Assume that $\psi$ satisfies the $\nabla_{2}$--condition. Then $A_{\alpha}^{\psi}\left(\mathbb{B}_{N}\right)$ can be isometrically identified with $\left(AM_{\alpha}^{\psi}\left(\mathbb{B}_{N}\right)\right)^{**}$. Moreover, $C_{\phi}$ is equal to the biadjoint of $C_{\phi}|_{AM_{\alpha}^{\psi}\left(\mathbb{B}_{N}\right)}$.
\end{enumerate}
\end{theorem}
Note that when $\psi(x)=x^{p}$, $1\leq p<+\infty$, the previous
theorem reminds that the classical Hardy and Bergman spaces are separable
and reflexive, a situation which no longer occurs when
the Orlicz function does not satisfy the $\Delta_2$--condition. Nevertheless it also points out some
similarities with Bloch-type spaces or weighted Banach spaces, and more generally with $M$--ideals \cite{perfekt2017}.

\medskip{}

We finish this paragraph by recalling sharp estimates for the point evaluation functionals on $A_{\alpha}^{\psi}\left(\mathbb{B}_{N}\right)$, $\alpha\geq-1$.
\begin{proposition}[\cite{charpentier_composition_2011,charpentier_composition_2013}]
\label{prop|point-evaluation-bounded-prop}Let $\alpha\geq-1$ and let $\psi$ be an Orlicz function. For any $a\in \B_N$, the point evaluation functional $\delta_{a}$ at $a$ is bounded on $A_{\alpha}^{\psi}\left(\mathbb{B}_{N}\right)$
and we have
\[
\frac{1}{4^{N+\alpha+1}}\psi^{-1}\left(\left(\frac{1+\left|a\right|}{1-\left|a\right|}\right)^{N+\alpha+1}\right)\leq\left\Vert \delta_{a}\right\Vert _{\left(A_{\alpha}^{\psi}\right)^{*}}=\left\Vert \delta_{a}\right\Vert _{\left(AM_{\alpha}^{\psi}\right)^{*}}\leq\psi^{-1}\left(\left(\frac{1+\left|a\right|}{1-\left|a\right|}\right)^{N+\alpha+1}\right).
\]

\medskip{}

\end{proposition}

\subsection{A first characterization of small weighted Bergman-Orlicz spaces}\label{first-carac-parag}

We start with the following observation.
\begin{lemma}
\label{lem|lemma-base-BO-equal-Growth}Let $\alpha>-1$ and $\psi$
be an Orlicz function satisfying the $\Delta^{2}$--condition. Then
the function $\psi^{-1}\left(\frac{1}{\left(1-|z|\right)^{a}}\right)$
belongs to $L^{\psi}\left(\mathbb{B}_{N},v_{\alpha}\right)$ for every
$a>0$.\end{lemma}
\begin{proof}
We recall that $\left(1-|z|\right)^{\beta}$ belongs to $L^{1}\left(\mathbb{B}_{N},v\right)$
if and only if $\beta>-1$. Then
it is enough to show that for any $a>0$, there exists some constant
$C>0$ and some $\beta>-1$ such that
\[
\psi\left(C\psi^{-1}\left(\frac{1}{\left(1-|z|\right)^{a}}\right)\right)\left(1-|z|\right)^{\alpha}\leq\left(1-|z|\right)^{\beta}
\]
for every $\left|z\right|$ close enough to $1$. Since $\psi$
satisfies the $\Delta^{2}$--condition, Proposition \ref{prop|prop-Orlicz-growth}
(3) ensures that for any $b>1$ there exists $C>0$ such
that $\psi\left(x\right)^{b}\leq \psi\left(Cx\right)$ for
any $x$ large enough. Using that $\psi$ is increasing, the last
inequality becomes $\psi\left(C^{-1}\psi^{-1}\left(\psi(x)^{b}\right)\right)\leq\psi(x)$,
and setting $y=\left(\psi(x)\right)^{b}$ gives $\psi\left(C^{-1}\psi^{-1}(y)\right)\leq y^{1/b}$
for any $y$ large enough. Thus, replacing $y$ with $1/\left(1-|z|\right)^{a}$,
we obtain
\[
\psi\left(C^{-1}\psi^{-1}\left(\frac{1}{\left(1-|z|\right)^{a}}\right)\right)\left(1-|z|\right)^{\alpha}\leq \left(1-|z|\right)^{\alpha-a/b}
\]
for every $\left|z\right|$ close enough to $1$. The lemma follows
by choosing $b$ large enough in order that $\alpha-a/b>-1$ (possible for $\alpha >-1$).
\end{proof}
Let us give another easy lemma.
\begin{lemma}
\label{lem|lemma-tech-charac-BO}Let $\psi$ be an Orlicz function
satisfying the $\Delta^{2}$--condition. For every $a\geq1$, there exists
a constant $C>0$ such that $\psi^{-1}\left(y\right)\leq\psi^{-1}\left(y^{a}\right)\leq C\psi^{-1}\left(y\right)$
for any $y$ large enough.\end{lemma}
\begin{proof}
The first inequality is valid for any $y\geq1$ since $\psi^{-1}$
is increasing and $a\geq 1$. For the second inequality, for $b>1$ let $C>0$
be given by Proposition \ref{prop|prop-Orlicz-growth} (3). We have
$\psi(x)^{b}\leq \psi\left(Cx\right)$ for any $x$ large
enough, hence setting $y=\left(\psi(x)\right)^{b/a}$ as
in the previous proof and composing the last inequality by the increasing
function $\psi^{-1}$, we get $\psi^{-1}(y^{a})\leq C\psi^{-1}\left(y^{a/b}\right)$
for any $y$ large enough. To finish we choose
$b$ so that $a/b$ is smaller than $1$ and we use again that $\psi^{-1}$
is increasing.
\end{proof}

\medskip{}

For any Orlicz function $\psi$ and and $\gamma \geq 1$, let us define the following typical radial weight:
\begin{equation}\label{weightpsi}
w^{\psi,\gamma}(z)=\left[\psi^{-1}\left(\frac{1}{\left(1-|z|\right)^{\gamma}}\right)\right]^{-1},\, z\in \B_N.
\end{equation}
When $\gamma =1$, we will denote $w^{\psi}=w^{\psi,\gamma}$. Proposition \ref{prop|point-evaluation-bounded-prop} tells that $A_{\alpha}^{\psi}(\B_N)\subset H_{w^{\psi,N+\alpha+1}}^{\infty}(\B_N)$ for every $\psi$ and every $\alpha \geq -1$. Moreover, Lemma \ref{lem|lemma-base-BO-equal-Growth} ensures that if $\psi \in \Delta ^2$ and $\alpha > -1$, then any $f$ in $H_{w^{\psi,N+\alpha+1}}^{\infty}\left(\mathbb{B}_{N}\right)$ lies also in $A_{\alpha}^{\psi}\left(\mathbb{B}_{N},v_{\alpha}\right)$. We deduce that, whenever $\psi \in \Delta ^2$ and $\alpha > -1$, $A_{\alpha}^{\psi}\left(\mathbb{B}_{N},v_{\alpha}\right)= H_{w^{\psi,N+\alpha+1}}^{\infty}\left(\mathbb{B}_{N}\right)$ and finally by Lemma \ref{lem|lemma-tech-charac-BO} that $A_{\alpha}^{\psi}\left(\mathbb{B}_{N},v_{\alpha}\right)= H_{w^{\psi}}^{\infty}\left(\mathbb{B}_{N}\right)$ (see \eqref{growth-space} in the introduction for the definition of the weighted Banach space $H_v^{\infty}(\B_N)$). Actually
the following theorem tells us a bit more.

Let us recall first that the \emph{vanishing weighted Banach space} -or \emph{vanishing
growth space}- associated with $v$ is defined by
\[
H_{v}^{0}:=\left\{ f\in H\left(\mathbb{B}_{N}\right),\,\lim_{|z|\rightarrow1}\left|f(z)\right|v(z)=0\right\} ,
\]
and, endowed with the norm $\left\Vert \cdot\right\Vert _{v}$, is a closed subspace of $H_{v}^{\infty}$.
\begin{theorem}
\label{thm|THM1}Let $\alpha>-1$ and let $\psi$ be an Orlicz function
satisfying the $\Delta^{2}$--condition. Then the weighted Bergman-Orlicz
space $A_{\alpha}^{\psi}\left(\mathbb{B}_{N}\right)$ coincides with
the growth space $H_{w^{\psi}}^{\infty}\left(\mathbb{B}_{N}\right)$,
with an equivalent norm. Moreover $AM_{\alpha}^{\psi}\left(\mathbb{B}_{N}\right)=H_{w^{\psi}}^{0}\left(\mathbb{B}_{N}\right)$.\end{theorem}
\begin{proof}
It remains to prove that the norms of $A_{\alpha}^{\psi}\left(\mathbb{B}_{N}\right)$
and $H_{w^{\psi}}^{\infty}\left(\mathbb{B}_{N}\right)$ are equivalent
and that $AM_{\alpha}^{\psi}\left(\mathbb{B}_{N}\right)=H_{w^{\psi}}^{0}\left(\mathbb{B}_{N}\right)$.
Let us assume for a while that the first assertion holds. Then the topologies
of $A_{\alpha}^{\psi}\left(\mathbb{B}_{N}\right)$ and $H_{w^{\psi}}^{\infty}\left(\mathbb{B}_{N}\right)$
are the same and it is known that $AM_{\alpha}^{\psi}\left(\mathbb{B}_{N}\right)$
and $H_{w^{\psi}}^{0}\left(\mathbb{B}_{N}\right)$ are the closure
of the set of polynomials for this topology (Theorem \ref{thm|weak-star-topo-H-O}
and \cite{shields_bounded_1978}), hence the second assertion.

Let us now prove that the norms of $A_{\alpha}^{\psi}\left(\mathbb{B}_{N}\right)$
and $H_{w^{\psi}}^{\infty}\left(\mathbb{B}_{N}\right)$ are indeed
equivalent. According to Proposition \ref{prop|point-evaluation-bounded-prop}
and Lemma \ref{lem|lemma-tech-charac-BO} we have
\[
|f(z)|\left[\psi^{-1}\left(\frac{1}{1-|z|}\right)\right]^{-1}\leq2^{N+\alpha+1}C\left\Vert f\right\Vert _{A_{\alpha}^{\psi}}
\]
for any $f\in A_{\alpha}^{\psi}\left(\mathbb{B}_{N}\right)$ and any
$z\in\mathbb{B}_{N}$ (note that we have used the concavity of $\psi^{-1}$),
where $C>0$ is the constant given by Lemma \ref{lem|lemma-tech-charac-BO}
and which only depends on $\psi$ and $\alpha$. This gives $\left\Vert \cdot\right\Vert _{w^{\psi}}\lesssim\left\Vert \cdot\right\Vert _{A_{\alpha}^{\psi}}$.
For the other inequality, we simply observe that for $f\in A_{\alpha}^{\psi}(\B_N)=H_{w^{\psi}}^{\infty}(\B_N)$,
\[
\frac{\left|f(z)\right|}{\left\Vert f\right\Vert _{w^{\psi}}}=\frac{\left|f(z)\right|}{{\displaystyle \sup_{z\in\mathbb{B}_{N}}\left|f(z)w^{\psi}(z)\right|}}\leq\psi^{-1}\left(\frac{1}{1-|z|}\right),\, z\in\mathbb{B}_{N}.
\]
By Lemma \ref{lem|lemma-base-BO-equal-Growth} $\psi^{-1}\left(\frac{1}{1-|z|}\right)$
is in $L^{\psi}\left(\mathbb{B}_{N},v_{\alpha}\right)$ so there exists
a constant $C'>0$ (depending only on $\psi$ and $\alpha$) such that
\[
\int_{\mathbb{B}_{N}}\psi\left(\frac{\left|f(z)\right|}{C'\left\Vert f\right\Vert _{w^{\psi}}}\right)dv_{\alpha}(z)\leq\int_{\mathbb{B}_{N}}\psi\left(\frac{1}{C'}\psi^{-1}\left(\frac{1}{1-|z|}\right)\right)dv_{\alpha}(z)\leq1.
\]
The expected inequality $\left\Vert f\right\Vert _{A_{\alpha}^{\psi}}\leq C'\left\Vert f\right\Vert _{w^{\psi}}$
follows.\end{proof}
\begin{unremark}
1) We will show in the next section that the $\Delta^{2}$--condition actually characterizes those Orlicz functions $\psi$ for which the associated weighted Bergman-Orlicz space $A_{\alpha}^{\psi}(\B_N)$ coincides with the weighted Banach space $H_{w^{\psi}}^{\infty}(\B_N)$. We will also see that this result does not hold for Hardy-Orlicz
spaces.

2) Note that, by Theorem \ref{thm|THM1}, $A_{\alpha}^{\psi}\left(\mathbb{B}_{N}\right)=A_{\beta}^{\psi}\left(\mathbb{B}_{N}\right)$
for any $\alpha,\beta>-1$ whenever the Orlicz function $\psi$ satisfies
the $\Delta^{2}$--condition; in other words, the corresponding weighted
Bergman-Orlicz spaces coincide with the non-weighted ones.

3) We shall recall Kaptano\u{g}lu's result \cite{kaptanoglu_carleson_2007} telling that $\mathcal{B}^{\gamma}\left(\mathbb{B}_{N}\right)=H_{v^{\gamma}}^{\infty}\left(\mathbb{B}_{N}\right)$
and $\mathcal{B}_{0}^{\gamma}\left(\mathbb{B}_{N}\right)=H_{v^{\gamma}}^{0}\left(\mathbb{B}_{N}\right)$
whenever $\gamma>0$, where $v(z)=1-|z|$ and $\mathcal{B}^{\gamma}\left(\mathbb{B}_{N}\right)$
stands for the weighted Bloch space. Theorem \ref{thm|THM1} completes
his description for the weighted Banach spaces with slow growth.
\end{unremark}
The following corollary immediately proceeds from Proposition \ref{prop|point-evaluation-bounded-prop}
and Lemma \ref{lem|lemma-base-BO-equal-Growth}.
\begin{corollary}
\label{cor|coro-evaluation-norm-Lpsi-Delta2}Let $\alpha>-1$ and
$\psi$ be an Orlicz function satisfying the $\Delta^{2}$--condition.
Then the function $z\mapsto\left\Vert \delta_{z}\right\Vert _{\left(A_{\alpha}^{\psi}\right)^{*}}$
belongs to $L^{\psi}\left(\mathbb{B}_{N},v_{\alpha}\right)$.
\end{corollary}
The previous provides another characterization of weighted Bergman-Orlicz
spaces associated with $\Delta^{2}$ Orlicz functions.
\begin{corollary}
Let $\alpha>-1$ and let $\psi$ be an Orlicz function satisfying the $\Delta^{2}$--condition. $A_{\alpha}^{\psi}\left(\mathbb{B}_{N}\right)$
coincides with the set of all functions $f$ holomorphic on $\mathbb{B}_{N}$
such that there exists $A>0$ such that
\begin{equation}
c_{\alpha}\int_{0}^{1}\sup_{\zeta\in\mathbb{S}_{N}}\psi\left(\frac{\left|f(r\zeta)\right|}{A}\right)r^{2N-1}\left(1-r^{2}\right)^{\alpha}dr<+\infty.\label{eq|charac-BO-sup}
\end{equation}
Moreover, if we define $\left\Vert \cdot\right\Vert _{A_{\alpha,\infty}^{\psi}}$
as the infimum of those constants $A>0$ such that the left hand-side
of (\ref{eq|charac-BO-sup}) is less than or equal to $1$, then $\left\Vert \cdot\right\Vert _{A_{\alpha,\infty}^{\psi}}$
is a norm on $A_{\alpha}^{\psi}\left(\mathbb{B}_{N}\right)$ and we
have
\begin{equation}
\left\Vert f\right\Vert _{A_{\alpha}^{\psi}}\leq\left\Vert f\right\Vert _{A_{\alpha,\infty}^{\psi}}\leq K\left\Vert f\right\Vert _{A_{\alpha}^{\psi}},\label{eq|eq_normes-Delta-up-2}
\end{equation}
for any $f\in A_{\alpha}^{\psi}\left(\mathbb{B}_{N}\right)$, and
some constant $K>0$.\end{corollary}
\begin{proof}
According to Corollary \ref{cor|coro-evaluation-norm-Lpsi-Delta2},
the maximum modulus principle and an integration in polar coordinates,
there exists a function $g\in L^{\psi}\left([0,1],c_{\alpha}r^{2N-1}\left(1-r^{2}\right)^{\alpha}dr\right)$
such that for any $f\in A_{\alpha}^{\psi}\left(\mathbb{B}_{N}\right)$,
\begin{equation}
\sup_{\left|z\right|\leq r}\left|f(z)\right|\leq g(r)\left\Vert f\right\Vert _{A_{\alpha}^{\psi}},\label{eq|carac-BO-sup-0}
\end{equation}
hence the first assertion. For the last assertions, observe first
that the inequality $\left\Vert f\right\Vert _{A_{\alpha}^{\psi}}\leq\left\Vert f\right\Vert _{A_{\alpha,\infty}^{\psi}}$
is trivial for any $f$ holomorphic on $\mathbb{B}_{N}$. Let now
$A>0$ and $f\in A_{\alpha}^{\psi}\left(\mathbb{B}_{N}\right)$. By
(\ref{eq|carac-BO-sup-0}) we get
\[
\int_{0}^{1}\sup_{\zeta\in\mathbb{S}_{N}}\psi\left(\frac{\left|f(r\zeta)\right|}{A}\right)r^{2N-1}\left(1-r^{2}\right)^{\alpha}dr\leq\int_{0}^{1}\psi\left(\frac{g(r)\left\Vert f\right\Vert _{A_{\alpha}^{\psi}}}{A}\right)r^{2N-1}\left(1-r^{2}\right)^{\alpha}dr.
\]
Let $K$ be the norm of $g$ in $L^{\psi}\left([0,1],c_{\alpha}r^{2N-1}\left(1-r^{2}\right)^{\alpha}dr\right)$.
If $A>\left\Vert f\right\Vert _{A_{\alpha}^{\psi}}K$, then the left
hand-side of the previous inequality is less than or equal to $1$,
so $A\geq\left\Vert f\right\Vert _{A_{\alpha,\infty}^{\psi}}$ hence
(\ref{eq|eq_normes-Delta-up-2}).

The fact that $\left\Vert \cdot\right\Vert _{A_{\alpha,\infty}^{\psi}}$
is a norm is easily checked.
\end{proof}
\smallskip{}

\begin{remark}
Notice that the constant $K$ above can be taken equal to the norm
in $L_{\alpha}^{\psi}$ of $z\mapsto\left\Vert \delta_{z}\right\Vert _{(A_{\alpha}^{\psi})^*}$,
up to some constant depending only on $N$ and $\alpha$.
\end{remark}
%

In the next section we will apply those observations to composition
operators, and furthermore we will be able to refine Theorem \ref{thm|THM1}. Of course, in view of Theorem \ref{thm|THM1}, the forthcoming results involving $A_{\alpha}^{\psi}(\B_N)$ with $\psi \in \Delta^2$, can be re-read with $H_{w^{\psi}}^{\infty}(\B_N)$ instead of $A_{\alpha}^{\psi}(\B_N)$.

\section{Application to composition operators}

In the sequel we will use the notations $w^{\psi,\gamma}$ and $w^{\psi}$, see \eqref{weightpsi}.

\subsection{Preliminary results}

\subsubsection{Composition operators on weighted Bergman-Orlicz and Hardy-Orlicz
spaces of $\mathbb{B}_{N}$}

A useful tool to study the boundedness and the compactness of composition
operators acting on Hardy-Orlicz and Bergman-Orlicz spaces are Carleson
embedding theorems. In this paper we will use these techniques only to deal with the compactness. They involve geometric notions that we will briefly
introduce here. For $\zeta\in\B_{N}$ and $0<h<1$, let us
denote by $S\left(\zeta,h\right)$ and $\mathcal{S}\left(\zeta,h\right)$
the non-isotropic ``balls'', respectively in $\mathbb{B}_{N}$ and
$\overline{\mathbb{B}_{N}}$, defined by
\[
S\left(\zeta,h\right)=\left\{ z\in\mathbb{B}_{N},\,\left|1-\left\langle z,\zeta\right\rangle \right|<h\right\} \mbox{ and }\mathcal{S}\left(\zeta,h\right)=\left\{ z\in\overline{\mathbb{B}_{N}},\,\left|1-\left\langle z,\zeta\right\rangle \right|<h\right\} .
\]
For $\phi:\mathbb{B}_{N}\rightarrow\mathbb{B}_{N}$, we denote by
$\mu_{\phi,\alpha}$ the pull-back measure of $v_{\alpha}$ by $\phi$
and by $\mu_{\phi}$ that of $\sigma_{N}$ by the boundary limit $\phi^{*}$
of $\phi$. To be precise, given $E\subset\mathbb{B}_{N}$ (resp.
$E\subset\overline{\mathbb{B}_{N}}$),
\[
\mu_{\phi,\alpha}\left(E\right)=v_{\alpha}\left(\phi^{-1}\left(E\right)\right)\mbox{ (resp. }\mu_{\phi}\left(E\right)=\sigma_{N}\left(\left(\phi^{*}\right)^{-1}\left(E\right)\cap\mathbb{S}_{N}\right)\mbox{)}.
\]
In view of the end of Paragraph 2.1, we will unify the notations denoting
also by $\mu_{\phi,\alpha}$, with $\alpha=-1$, the measure $\mu_{\phi}$
(the latter acting on $\overline{\mathbb{B}_{N}}$). Then we define
the functions $\varrho_{\phi,\alpha}$ on the interval $\left(0,1\right)$ by
\[
\varrho_{\phi,\alpha}\left(h\right)=\left\{ \begin{array}{l}
\sup_{\zeta\in\mathbb{S}_{N}}\mu_{\phi,\alpha}\left(S\left(\zeta,h\right)\right)\mbox{ if }\alpha>-1\\
\sup_{\zeta\in\mathbb{S}_{N}}\mu_{\phi,\alpha}\left(\mathcal{S}\left(\zeta,h\right)\right)\mbox{ if }\alpha=-1
\end{array}\right..
\]

First of all, by testing the boundedness of $C_{\phi}$ on a standard function, one easily gets the following necessary condition.

\begin{fact}\label{factbounded}If $C_{\phi}$ is bounded on $A_{\alpha}^{\psi}\left(\mathbb{B}_{N}\right)$
then there exist $A>0$ and $0<\eta <1$, such that
\[
\varrho_{\phi,\alpha}(h)\leq\frac{1}{\psi\left(A\psi^{-1}\left(1/h^{N+\alpha+1}\right)\right)}
\]
for every $h\in\left(0,\eta\right)$.
\end{fact}

The next theorem, stated for $\psi$ in the $\Delta^2$--class, is a particular case of the one obtained in \cite{charpentier_composition_2011,charpentier_composition_2013} for a larger class of Orlicz functions (see also \cite{lefevre_compact_2009,lefevre_compact_2013} for the unit disc).
\begin{theorem}
\label{thm|carac-cont-embedding-H-O}Let $\alpha\geq-1$, let $\psi$
be an Orlicz function satisfying the $\Delta^2$--condition, and let $\phi:\mathbb{B}_{N}\rightarrow\mathbb{B}_{N}$
be holomorphic. The composition operator $C_{\phi}$ is compact from $A_{\alpha}^{\psi}$ into itself if and only if for every $A>0$, there exists $h_A \in (0,1)$ such that
\[
\varrho_{\phi,\alpha}(h)\leq\frac{1}{\psi\left(A\psi^{-1}\left(1/h^{N+\alpha+1}\right)\right)}
\]
for every $h\in\left(0,h_A\right)$.
\end{theorem}

\medskip{}

For $\alpha\geq-1$ and $a\in\mathbb{B}_{N}$, we define the function
$f_{a}\in H^{\infty}$ by
\[
f_a(z)=\left(1-\left|a\right|\right)^{N+\alpha+1}\left(\frac{1-\left|a\right|^{2}}{\left(1-\left\langle z,a\right\rangle \right)^{2}}\right)^{N+\alpha+1}.
\]
Then $|f_{a}\left(z\right)|=\left(1-\left|a\right|\right)^{N+\alpha+1}H_{a}\left(z\right)$ where $H_{a}$ is the \emph{Berezin kernel}. The following corollary will
be useful for the proof of (2)$\Leftrightarrow$(3) in Theorem \ref{main-thm-compacntess-Deltaup2}.
\begin{corollary}
\label{cor|corollaire_Berezin_cont_comp}Let $\alpha\geq-1$, let
$\psi$ be an Orlicz function satisfying the $\Delta^{2}$--condition
and let $\phi:\mathbb{B}_{N}\rightarrow\mathbb{B}_{N}$ be holomorphic.
$C_{\phi}$ is compact from $A_{\alpha}^{\psi}\left(\mathbb{B}_{N}\right)$
into itself if and only if
\begin{equation}\label{fa-comp-compo-op}
\left\Vert f_{a}\circ\phi\right\Vert _{A_{\alpha}^{\psi}\left(\mathbb{B}_{N}\right)}=o_{\left|a\right|\rightarrow1}\left(\frac{1}{\psi^{-1}\left(1/\left(1-\left|a\right|\right)^{N+\alpha+1}\right)}\right).
\end{equation}

\end{corollary}
\begin{unremark}This result was stated in \cite{lefevre_compact_2009} for an arbitrary
Orlicz function, when $\alpha=-1$ and $N=1$. Theorem \ref{thm|carac-cont-embedding-H-O}, Corollary \ref{cor|corollaire_Berezin_cont_comp} and its proof below, work for a class of
Orlicz functions much larger than the $\Delta^{2}$--class, namely for
Orlicz functions satisfying the so-called $\nabla_{0}$--condition; we refer
to \cite{charpentier_composition_2011,charpentier_composition_2013}.
\end{unremark}
\begin{proof}[Proof of the Corollary]For the \emph{only if} part, we introduce
the function
\[
g_{a}\left(z\right)=\frac{\psi^{-1}\left(1/\left(1-\left|a\right|\right)^{N+\alpha+1}\right)}{2^{N+\alpha+1}}f_{a}\left(z\right),\, z\in\mathbb{B}_{N}.
\]
$g_{a}$ lies in the unit ball of $A_{\alpha}^{\psi}\left(\mathbb{B}_{N}\right)$
(see \cite[Lemma 3.9]{lefevre_compact_2009} and the proofs of \cite[Propositions 1.9 and 1.6 respectively]{charpentier_composition_2011,charpentier_composition_2013})
and tends to $0$ uniformly on every compact set of $\mathbb{B}_{N}$,
as $\left|a\right|$ tends to $1$. In particular $g_{a}\circ\phi$
converges pointwise to $0$ as $\left|a\right|\rightarrow1$. $C_{\phi}$
being compact on $A_{\alpha}^{\psi}\left(\mathbb{B}_{N}\right)$,
up to take a subsequence, we may assume that $g_{a}\circ\phi$ tends
in $A_{\alpha}^{\psi}\left(\mathbb{B}_{N}\right)$ to some function
$g$ as $\left|a\right|\rightarrow1$, hence $g=0$ since the convergence
in $A_{\alpha}^{\psi}\left(\mathbb{B}_{N}\right)$ implies the pointwise
one.
This ends the proof of the only if part.

For the converse, we only deal with $\alpha > -1$, the case $\alpha = -1$ being the same up to change the non-isotropic balls. It follows from the proof of \cite[Theorem 2.5 (1)]{charpentier_composition_2013} that, given $\zeta \in \mathbb{S}_N$ and $a=|a|\zeta \in \B_N$, the following estimate holds:
$$\mu_{\phi,\alpha}(S(\zeta, 1-|a|))\leq \frac{1}{\psi\left(K/\left\Vert f_a\circ \phi\right\Vert _{A_{\alpha}^{\psi}}\right)}$$
for some constant $K>0$ independant of $\zeta$. So, if \eqref{fa-comp-compo-op} holds, then for every $\varepsilon > 0$, there exists $h_{\varepsilon}\in (0,1)$ such that for every $\zeta \in \mathbb{S}_N$ and  every $a=|a|\zeta \in \B_N$ with $|a|\in (1-h_{\varepsilon},1)$ we have
\begin{equation}\label{todaytrice}\mu_{\phi,\alpha}(S(\zeta, 1-|a|))\leq \frac{1}{\psi\left(\frac{K}{\varepsilon}\psi^{-1}\left(1/\left(1-\left|a\right|\right)^{N+\alpha+1}\right)\right)},\end{equation}
hence $C_{\phi}$ is compact on $A_{\alpha}^{\psi}(\B_N)$ by Theorem \ref{thm|carac-cont-embedding-H-O}.
\end{proof}
\smallskip{}

We deduce from the previous corollary some sufficient conditions for the compactness of $C_{\phi}$ on $A_{\alpha}^{\psi}(\B_N)$.

\begin{proposition}
\label{prop|1st_applic}Let $\alpha\geq-1$, let $\psi$ be an Orlicz
function satisfying the $\Delta^2$--condition and let $\phi:\mathbb{B}_{N}\rightarrow\mathbb{B}_{N}$
be holomorphic. Then $C_{\phi}$ is compact on $A_{\alpha}^{\psi}\left(\mathbb{B}_{N}\right)$ whenever one of the two following conditions is satisfied:
\begin{enumerate}
\item [(i)]$1/\left(1-\left|\phi\right|\right) \in L_{\alpha}^{\psi}$;
\item [(ii)]$\sum_{n=0}^{\infty}\left\Vert |\phi|^{n}\right\Vert _{L_{\alpha}^{\psi}}<\infty$.
\end{enumerate}
\end{proposition}

\begin{proof}Observe that
\begin{multline*}
\limsup_{\left|a\right|\rightarrow1}\psi^{-1}\left(\frac{1}{\left(1-\left|a\right|\right)^{N+\alpha+1}}\right)\left\Vert f_{a}\circ\phi\right\Vert _{A_{\alpha}^{\psi}}\\\lesssim
\limsup_{\left|a\right|\rightarrow1}\psi^{-1}\left(\frac{1}{\left(1-\left|a\right|\right)^{N+\alpha+1}}\right)\left(1-\left|a\right|\right)^{N+\alpha+1}\left\Vert \frac{1}{\left(1-\left|\phi\right|\right)^{N+\alpha+1}}\right\Vert _{L_{\alpha}^{\psi}}.\label{eq|1st_eq_applications}
\end{multline*}
Since $\psi$ is an Orlicz function,
\[
\limsup_{\left|a\right|\rightarrow1}\psi^{-1}\left(\frac{1}{\left(1-\left|a\right|\right)^{N+\alpha+1}}\right)\left(1-\left|a\right|\right)^{N+\alpha+1}=0,
\]
and, by Corollary \ref{cor|corollaire_Berezin_cont_comp}, $C_{\phi}$ is
compact on $A_{\alpha}^{\psi}\left(\mathbb{B}_{N}\right)$, $\alpha \geq -1$, whenever $1/\left(1-\left|\phi\right|\right)^{N+\alpha+1}$ is in $L_{\alpha}^{\psi}$.
Now since $\psi$ satisfies the $\Delta^{2}$--condition, Lemma \ref{lem|lemma-tech-charac-BO} makes it clear that (i) is indeed sufficient for the compactness of $C_{\phi}$. To deal with (ii), we may just notice that
$$\left\Vert \frac{1}{1-\left|\phi\left(z\right)\right|}\right\Vert _{L_{\alpha}^{\psi}} \leq \sum_{n=0}^{\infty}\left\Vert |\phi|^{n}\right\Vert _{L_{\alpha}^{\psi}}.$$
\end{proof}

\begin{unremark}(1) Like Theorem \ref{thm|carac-cont-embedding-H-O} and Corollary \ref{cor|corollaire_Berezin_cont_comp}, the previous proposition actually holds for Orlicz functions satisfying the so-called $\nabla_0$--condition.\\
(2) In Section 3.4, we shall see some conditions similar to (i) or (ii) above, but necessary and sufficient for the compactness of $C_{\phi}$ on $A_{\alpha}^{\psi}(\B_N)$ when $\psi$ satisfies the $\Delta^2$--condition.
\end{unremark}

\subsubsection{Composition operators on weighted Banach spaces of $\mathbb{B}_{N}$}
In \cite{bonet_composition_1998}, the authors give a mild condition on a typical weight
$v$ on the unit disc $\mathbb{D}$ to ensure that every composition
operators acting the weighted Banach space $H_{v}^{\infty}\left(\mathbb{D}\right)$
is bounded. We provide below a straightforward proof of this fact
for the specific weights $w^{\psi,\gamma}$ on the unit ball $\mathbb{B}_{N}$.
We need two lemmas and some definition.

For a typical weight $v$ on $\mathbb{B}_{N}$, let $B_{v}={\displaystyle \left\{ f\in H(\mathbb{B}_{N});\,\left|f\right|\leq\frac{1}{v}\right\} }$ denote the
unit ball of $H_{v}^{\infty}\left(\mathbb{B}_{N}\right)$. We define
$\widetilde{v}(z)=\left(\sup_{\left\Vert f\right\Vert _{v}<1}\left\{ \left|f(z)\right|\right\}\right)^{-1}$, $z\in\mathbb{B}_{N},$
and say that $v$ is \emph{essential} if there exist two constants $c$ and $C$ such that
$$c\widetilde{v}(z)\leq v(z) \leq C\widetilde{v}(z)$$
for any $z\in \B_N$.

\begin{lemma}
\label{lem|weight-essential-ball-spec-case}For $N\geq 1$, $\gamma \geq N$ and $\psi$ an Orlicz function, the weight $w^{\psi,\gamma}$ is essential.
\end{lemma}
\begin{proof}
The inequality $w^{\psi,\gamma}(z)\leq \widetilde{w^{\psi,\gamma}}(z)$ follows from the definition of $\widetilde{w^{\psi,\gamma}}(z)$. For the other one, let us choose $\alpha \geq -1$ such that $\gamma =N+\alpha+1$ (this is possible since $\gamma \geq N$). We now use Proposition
\ref{prop|point-evaluation-bounded-prop} to observe first that $\left\Vert \cdot\right\Vert _{w^{\psi,\gamma}}\leq 2^{N+\alpha+1}\left\Vert \cdot\right\Vert _{A_{\alpha}^{\psi}}$
and second that for every $z\in\mathbb{B}_{N}$, there exists $f$ in the unit ball of $A_{\alpha}^{\psi}\left(\mathbb{B}_{N}\right)$
such that $\left|f(z)\right|\geq\frac{1}{4^{N+\alpha+1}w^{\psi,\gamma}(z)}$;
hence the function $g:=f/2^{N+\alpha+1}$ belongs to the unit ball $B_{w^{\psi,\gamma}}$ of $H_{w^{\psi,\gamma}}^{\infty}$ and satisfies $\left|g(z)\right|\geq\frac{1}{8^{N+\alpha+1}w^{\psi,\gamma}(z)}$.
Then for every $z\in\mathbb{B}_{N}$, $\frac{1}{\widetilde{w^{\psi,\gamma}}(z)}\geq\frac{1}{8^{N+\alpha+1}w^{\psi,\gamma}(z)}$,
as expected.
%
\end{proof}
Note that the previous proof underlines the fact that a weight $v$ is essential if and only if there exists a Banach space of holomorphic functions $X$ on $\B_N$, continuously embedded into $H_v^{\infty}(\B_N)$, such that for any $z\in \B_N$, $1/v(z)$ is the norm of the evaluation at $z$ on $X$, up to some constants independant of $z$.

More generally, some conditions for a weight on $\mathbb{D}$ to be essential were exhibited in \cite{bierstedt_associated_1998} and, more recently, some general characterizations of such weights have been given in \cite{abakumov_moduli_2015}.
\begin{lemma}
\label{lem|lemme-sup-psi-frac-bounded}Let $\psi$ be an Orlicz function,
$\gamma\geq1$ and $\phi$ a holomorphic self-map of $\mathbb{B}_{N}$.
Then
\[
\sup_{z\in\mathbb{B}_{N}}\frac{\psi^{-1}\left({\displaystyle \frac{1}{\left(1-\left|\phi(z)\right|\right)^{\gamma}}}\right)}{\psi^{-1}\left({\displaystyle \frac{1}{\left(1-\left|z\right|\right)^{\gamma}}}\right)}<+\infty.
\]
\end{lemma}
\begin{proof}
Let $a=\phi(0)$ and let $\varphi_{a}$ be an automorphism of $\mathbb{B}_{N}$
which vanishes at $a$. Since $\varphi_{a}$ is an involution, $\phi=\varphi_{a}\circ\varphi_{a}\circ\phi$,
and $\varphi_{a}\circ\phi(0)=0$. Recall that for any automorphism
$\varphi_{a}$ of $\mathbb{B}_{N}$, we have
\[
1-\left|\varphi_{a}(z)\right|^{2}=\frac{\left(1-\left|a\right|^{2}\right)\left(1-\left|z\right|^{2}\right)}{\left|1-\left\langle a,z\right\rangle \right|^{2}}
\]
for every $z\in\mathbb{B}_{N}$ \cite[Theorem 2.2.2]{rudin_function_1980}. Then,
applying the Schwarz lemma to $\varphi_{a}\circ\phi$ and using the
convexity of $\psi$, it follows that
\begin{eqnarray*}
\psi^{-1}\left({\displaystyle \frac{1}{\left(1-\left|\phi(z)\right|\right)^{\gamma}}}\right) & \leq & 2^{\gamma}\psi^{-1}\left(\frac{1}{\left(1-\left|\varphi_{a}\circ\varphi_{a}\circ\phi(z)\right|^{2}\right)^{\gamma}}\right)\\
 & = & 2^{\gamma}\psi^{-1}\left(\left(\frac{\left|1-\left\langle a,\varphi_{a}\circ\phi(z)\right\rangle \right|^{2}}{\left(1-\left|a\right|^{2}\right)\left(1-\left|\varphi_{a}\circ\phi(z)\right|^{2}\right)}\right)^{\gamma}\right)\\
 & \leq & \left(\frac{2}{1-\left|a\right|}\right)^{\gamma}\psi^{-1}\left(\frac{1}{\left(1-\left|z\right|\right)^{\gamma}}\right).
\end{eqnarray*}

\end{proof}
Let $f\in H_{w^{\psi,\gamma}}^{\infty}\left(\mathbb{B}_{N}\right)$.
By Lemma \ref{lem|lemme-sup-psi-frac-bounded} we have $w^{\psi,\gamma}(z)\leq Cw^{\psi,\gamma}(\phi(z))$
for some constant $C>0$. Thus
\begin{equation}\label{todaybis}
\left|f\circ\phi(z)\right|w^{\psi,\gamma}(z)=\left|f\circ\phi(z)\right|w^{\psi,\gamma}(\phi(z))\frac{w^{\psi,\gamma}(z)}{w^{\psi,\gamma}(\phi(z))}\leq C\left\Vert f\right\Vert _{H_{w^{\psi,\gamma}}^{\infty}}
\end{equation}
and we deduce:
\begin{proposition}
\label{prop|prop-Cphi-bounded-growth-space}Let $\psi$ be an Orlicz
function and $\gamma\geq 1$. For every holomorphic self-map $\phi$ of $\mathbb{B}_{N}$, the composition operator $C_{\phi}$ is bounded on $H_{w^{\psi,\gamma}}^{\infty}\left(\mathbb{B}_{N}\right)$.
\end{proposition}
Since $A_{\alpha}^{\psi}\left(\mathbb{B}_{N}\right)$ coincides with
$H_{w^{\psi}}^{\infty}\left(\mathbb{B}_{N}\right)$ whenever $\psi$
satisfies the $\Delta^{2}$--condition (Theorem \ref{thm|THM1}), the first assertion of the
following corollary immediately proceeds from the previous proposition.
\begin{corollary}
\label{coro|Cphi-bounded-BO-Delta2}Let $\alpha>-1$ and $\psi$ be
an Orlicz function satisfying the $\Delta^{2}$--condition. For every holomorphic self-map $\phi$ of $\B_N$, $C_{\phi}$ is bounded on $A_{\alpha}^{\psi}\left(\mathbb{B}_{N}\right)$. Moreover
\begin{equation}\label{norm-cphi-small}\left\Vert C_{\phi}\right\Vert _{A_{\alpha}^{\psi}\rightarrow A_{\alpha}^{\psi}}\simeq \sup_{z\in\mathbb{B}_{N}}\frac{\psi^{-1}\left({\displaystyle \frac{1}{\left(1-\left|\phi(z)\right|\right)^{N+\alpha+1}}}\right)}{\psi^{-1}\left({\displaystyle \frac{1}{\left(1-\left|z\right|\right)^{N+\alpha+1}}}\right)}\simeq \sup_{z\in\mathbb{B}_{N}}\frac{\psi^{-1}\left(\displaystyle \frac{1}{1-\left|\phi(z)\right|}\right)}{\psi^{-1}\left({\displaystyle \frac{1}{1-\left|z\right|}}\right)}.
\end{equation}
\end{corollary}

\begin{proof}Only \eqref{norm-cphi-small} needs to be checked. Now, starting with \eqref{todaybis}, the first estimate is a consequence of Theorem \ref{thm|THM1} and Lemma \ref{lem|weight-essential-ball-spec-case} (applied to $\gamma = N+\alpha+1$) and the second then follows from Lemma \ref{lem|lemma-tech-charac-BO}.
\end{proof}
\begin{unremark}(1) Using Carleson measure theorems, it was already stated in \cite{charpentier_composition_2011,charpentier_composition_2013}
that every composition operator is bounded on $A_{\alpha}^{\psi}\left(\mathbb{B}_{N}\right)$
and $H^{\psi}\left(\mathbb{B}_{N}\right)$ whenever $\psi$ satisfies
the $\Delta^{2}$--condition. Yet, whatever the Orlicz function, no explicit estimate of the norm of a composition operator where known for $N> 1$.\\
(2) We will see in the next paragraph that the estimates in \eqref{norm-cphi-small} does not hold whenever $\psi$ does not satisfy the $\Delta^2$--condition; indeed, while the right-hand side of the estimate is always bounded (Lemma \ref{lem|lemme-sup-psi-frac-bounded}), we will see that, if $\psi \notin \Delta^2$, then there exist unbounded composition operators.\\
(3) More generally, it is easily seen that the norm of $C_{\phi}$ on $H_v^{\infty}$, $v$ a typical radial weight, is equal to $\displaystyle{\sup_{z\in \B_N}\frac{v(z)}{\widetilde{v}(\phi(z))}}$, where $\widetilde{v}$ is defined similarly as in Lemma \ref{lem|weight-essential-ball-spec-case} (see \cite{bonet_essential_1999,bonet_composition_1998} for $N=1$) and can be replaced with $v$, up to some constant, whenever $v$ is essential.
\end{unremark}
In the next paragraph, we will show that the $\Delta^{2}$--condition
in fact characterizes these extreme behaviors.

\subsection{Order boundedness of composition operators}

Let us briefly recall that an operator $T$ from a Banach space $X$
to a Banach lattice $Y$ is order bounded into some subspace $Z$
of $Y$ if there exists $y\in Z$ positive such that $\left|Tx\right|\leq y$
for every $x$ in the unit ball of $X$. The following lemma is immediate and its proof is left to the reader.
\begin{lemma}\label{lem-OB-gene}Let $Y$ be a Banach lattice and $X$ a closed sublattice of $Y$. If the canonical embedding $I$ from $X$ into $Y$ is order-bounded into $Y$, then a linear map $T:X\rightarrow Y$ with $T(X)\subset X$ is order-bounded into $Y$ if and only if it is bounded (from $X$ into $X$).
\end{lemma}

In the sequel $X$ will stand for $A_{\alpha}^{\psi}\left(\mathbb{B}_{N}\right)$, $Y$ for $L_{\alpha}^{\psi}$ and $Z$ for either $Y$ or $M_{\alpha}^{\psi}$, with $\alpha\geq-1$. Corollary \ref{cor|coro-evaluation-norm-Lpsi-Delta2} can be reformulated
by saying that, if $\psi$ satisfies the $\Delta^2$--condition, then the canonical embedding $I$ from $A_{\alpha}^{\psi}\left(\mathbb{B}_{N},v_{\alpha}\right)$
into $L^{\psi}\left(\mathbb{B}_{N},v_{\alpha}\right)$ is order bounded
into $L^{\psi}\left(\mathbb{B}_{N},v_{\alpha}\right)$. Thus, Lemma \ref{lem-OB-gene} together with Corollary \ref{coro|Cphi-bounded-BO-Delta2} already yields:
\begin{theorem}
\label{thm|O-B-always_B-O}Let $\alpha>-1$ and let $\psi$ be an
Orlicz function satisfying the $\Delta^{2}$--condition. Every composition
operator acting on $A_{\alpha}^{\psi}\left(\mathbb{B}_{N}\right)$
is order bounded into $L^{\psi}\left(\mathbb{B}_{N},v_{\alpha}\right)$.\end{theorem}

\begin{remark}Since $H^{\psi}(\B_N)$ is canonically embedded into $A_{\alpha}^{\psi}(\B_N)$, it immediately follows from the previous result that the canonical injection $H^{\psi}(\B_N)\hookrightarrow A_{\alpha}^{\psi}(\B_N)$ is order bounded into $L_{\alpha}^{\psi}$ for any $\alpha >-1$. Yet, as we will see (Remark \ref{rem|rem-HO-growth}), the canonical embedding from $H^{\psi}(\B_N)$ into $L_{\alpha}^{\psi}$ with $\alpha =-1$ is not order-bounded into $L_{\alpha}^{\psi}$, and Theorem \ref{thm|O-B-always_B-O} does not hold for Hardy-Orlicz spaces.
\end{remark}

Moreover, proceeding as in \cite[page 18, the Remark]{lefevre_compact_2009}, the next proposition can be easily checked.

\begin{proposition}\label{OBMpsiimpliescomp}Let $\alpha\geq -1$ and let $\psi$ be an
Orlicz function. Every composition operator acting on $A_{\alpha}^{\psi}\left(\mathbb{B}_{N}\right)$ which is order bounded into $M_{\alpha}^{\psi}$ is compact.
\end{proposition}

\medskip{}

By definition, if $\delta_{z}$ stands for the functional of evaluation
at the point $z$ and $\phi$ denotes a holomorphic self-map of $\mathbb{B}_{N}$,
then the composition operator $C_{\phi}$ acting on $A_{\alpha}^{\psi}(\B_N)$,
$\alpha\geq-1$, is order bounded into $L_{\alpha}^{\psi}$ (resp.
$M_{\alpha}^{\psi}$) if and only if the function $z\mapsto\left\Vert \delta_{\phi(z)}\right\Vert _{\left(A_{\alpha}^{\psi}\right)^{*}}$
belongs to $L_{\alpha}^{\psi}$ (resp. $M_{\alpha}^{\psi}$). According
to Proposition \ref{prop|point-evaluation-bounded-prop}, we thus
get the following.
\begin{proposition}
\label{prop|prop-carac-OB-evaluation}Let $\alpha\geq-1$, let $\psi$
be an Orlicz function and let $\phi:\mathbb{B}_{N}\rightarrow\mathbb{B}_{N}$
be holomorphic.\begin{enumerate}\item $C_{\phi}$ acting on $A_{\alpha}^{\psi}\left(\mathbb{B}_{N}\right)$
is order-bounded into $L_{\alpha}^{\psi}$ if and only if
\begin{equation}
\psi^{-1}\left({\displaystyle \frac{1}{\left(1-\left|\phi\right|\right)^{N+\alpha+1}}}\right)\in L_{\alpha}^{\psi};\label{eq|cond-simple-OB-Lpsi}
\end{equation}
\item$C_{\phi}$ acting on $A_{\alpha}^{\psi}\left(\mathbb{B}_{N}\right)$
is order-bounded into $M_{\alpha}^{\psi}\left(\mathbb{B}_{N}\right)$
if and only if
\begin{equation}
\psi^{-1}\left({\displaystyle \frac{1}{\left(1-\left|\phi\right|\right)^{N+\alpha+1}}}\right)\in M_{\alpha}^{\psi}.\label{eq|cond-simple-OB-Mpsi}
\end{equation}
\end{enumerate}\end{proposition}
If $\psi(x)=x^{p}$, Assertions (1) and (2) are the same and we recall that, in this case, Condition \eqref{eq|cond-simple-OB-Lpsi} (or \eqref{eq|cond-simple-OB-Mpsi}) is equivalent to $C_{\phi}$ being Hilbert-Schmidt on $A_{\alpha}^{2}(\B_N)$ (see the introduction for more precisions).

\begin{remark}\label{today}Let $\psi$ be any Orlicz function and $\alpha \geq -1$. We immediately deduce from Proposition \ref{prop|prop-carac-OB-evaluation} the following observations:
\begin{enumerate}
\item [(i)]If $C_{\phi}:A_{\alpha}^p(\B_N)\rightarrow A_{\alpha}^p(\B_N)$ is order bounded into $L_{\alpha}^p$ for some $1\leq p< +\infty$ (equivalently Hilbert-Schmidt if $p=2$) then $C_{\phi}:A_{\alpha}^{\psi}(\B_N)\rightarrow A_{\alpha}^{\psi}(\B_N)$ is order bounded into $L_{\alpha}^{\psi}$;
\item [(ii)]If $C_{\phi}:A_{\alpha}^{\psi}(\B_N)\rightarrow A_{\alpha}^{\psi}(\B_N)$ is order bounded into $M_{\alpha}^{\psi}$ then $C_{\phi}:A_{\alpha}^p(\B_N)\rightarrow A_{\alpha}^p(\B_N)$ is order bounded into $L_{\alpha}^{p}$ (or equivalently into $M_{\alpha}^p$).
\end{enumerate}
\end{remark}

\medskip{}

In order to prove our main results, we need to complete Proposition \ref{prop|prop-carac-OB-evaluation} with a more geometric understanding of the order boundedness of composition operators. For $\alpha\geq-1$ and $h\in\left(0,1\right)$, let us denote by
$C_{\alpha}\left(h\right)$ the corona defined by
\begin{equation}
C_{\alpha}\left(h\right)=\left\{ \begin{array}{c}
\left\{ z\in\mathbb{B}_{N},\,1-\left|z\right|<h\right\} \mbox{ if }\alpha>-1\\
\left\{ z\in\overline{\mathbb{B}_{N}},\,1-\left|z\right|<h\right\} \mbox{ if }\alpha=-1
\end{array}\right..\label{eq|def-corona-alpha}
\end{equation}
We extend \cite[Theorem 3.15]{lefevre_compact_2009}
to dimension $N>1$ and all $\alpha \geq -1$, providing a useful alternative to Proposition
\ref{prop|prop-carac-OB-evaluation}.
\begin{theorem}
\label{thm|thm_2nd_applic}Let $\alpha\geq-1$, let $\psi$ be an
Orlicz function and let $\phi:\mathbb{B}_{N}\rightarrow\mathbb{B}_{N}$
be holomorphic.\begin{enumerate}\item \begin{enumerate}\item If
$C_{\phi}$ acting on $A_{\alpha}^{\psi}\left(\mathbb{B}_{N}\right)$
is order bounded, then there exist $A>0$ and $\eta \in (0,1)$ such that
\begin{equation}
\mu_{\phi}\left(C_{\alpha}\left(h\right)\right)\leq\frac{1}{\psi\left(A\psi^{-1}\left(1/h^{N+\alpha+1}\right)\right)}\label{eq|eq_order-bound-Cond}
\end{equation}
for any $h\in\left(0,\eta\right)$.\item If $\psi$ satisfies the $\Delta^{1}$--condition,
then (\ref{eq|eq_order-bound-Cond}) is sufficient for the order boundedness
of $C_{\phi}$.\end{enumerate}\item \begin{enumerate}\item If $C_{\phi}$
acting on $A_{\alpha}^{\psi}\left(\mathbb{B}_{N}\right)$ is order
bounded into $AM_{\alpha}^{\psi}\left(\mathbb{B}_{N}\right)$, then
for every $A>0$, there exist $C_{A}>0$ and $h_{A}\in\left(0,1\right)$
such that
\begin{equation}
\mu_{\phi}\left(C_{\alpha}\left(h\right)\right)\leq\frac{C_{A}}{\psi\left(A\psi^{-1}\left(1/h^{N+\alpha+1}\right)\right)}\label{eq|eq_order-bound-Mpsi-Cond}
\end{equation}
for any $h\in\left(0,h_{A}\right)$.\item If $\psi$ satisfies the
$\Delta^{1}$--condition, then (\ref{eq|eq_order-bound-Mpsi-Cond}) is
sufficient for the order boundedness of $C_{\phi}$ into $AM_{\alpha}^{\psi}\left(\mathbb{B}_{N}\right)$.\end{enumerate}\end{enumerate}
\end{theorem}

\begin{remark}
\label{rem|rem-HO-growth}By Theorem \ref{thm|thm_2nd_applic}, Theorem \ref{thm|O-B-always_B-O} is trivially
false for $\alpha=-1$, i.e. for $H^{\psi}\left(\mathbb{B}_{N}\right)$:
if $\phi\left(z\right)=z$ for every $z\in\mathbb{B}_{N}$, $\mu_{\phi}\left(C_{-1}\left(h\right)\right)=1$
hence $C_{\phi}$ is not order bounded.
Thus we also deduce that, whatever the Orlicz function $\psi$, $H^{\psi}\left(\mathbb{B}_{N}\right)$
does not coincide with any weighted Banach space.
\end{remark}

The proof of Theorem \ref{thm|thm_2nd_applic} is an adaptation of that of \cite[Theorem 3.15]{lefevre_compact_2009}. It relies on the introduction of the \emph{weak-}Orlicz space:
\begin{Definition}Given an Orlicz function $\psi$ and a probability space $(\Omega, \P)$, the \emph{weak-}Orlicz space $L^{\psi,\infty}(\Omega,\P)$ is the space of all measurable functions $f:\Omega \rightarrow \C$ such that, for some constant $c>0$, one has
$$\P(|f|>t)\leq \frac{1}{\psi(ct)},$$
for every $t>0$.
\end{Definition}
We then have the following \cite[Proposition 3.18]{lefevre_compact_2009} which is the key of the proof of \cite[Theorem 3.15]{lefevre_compact_2009}.
\begin{proposition}\label{prop-weak-Orlicz}If $\psi$ is an Orlicz function satisfying the $\Delta^1$--condition, then $L^{\psi}(\Omega,\P)=L^{\psi,\infty}(\Omega,\P)$.
\end{proposition}

\begin{proof}[Proof of Theorem \ref{thm|thm_2nd_applic}]For $\alpha >-1$, the both \emph{only if} parts proceed from Proposition \ref{prop|prop-carac-OB-evaluation} and the Markov's inequality
\begin{eqnarray*}\mu _{\phi}(C_{\alpha}(h)) & \leq & v_{\alpha}\left(\psi \left(A\psi^{-1}\left(\frac{1}{(1-|\phi|)^{N+\alpha+1}}\right)\right)>\psi \left(A\psi^{-1}\left(\frac{1}{h^{N+\alpha+1}}\right)\right)\right)\\
& \leq & \frac{1}{\psi \left(A\psi^{-1}\left(1/h^{N+\alpha+1}\right)\right)}\int _{\B_N}\psi \left(A\psi^{-1}\left(\frac{1}{(1-|\phi|)^{N+\alpha+1}}\right)\right)dv_{\alpha}.
\end{eqnarray*}
For $\alpha =-1$, it is enough to replace $\B_N$ by $\mathbb{S}_N$ in the previous inequalities (and to remind the notations).

We now assume that (\ref{eq|eq_order-bound-Cond}) (resp. (\ref{eq|eq_order-bound-Mpsi-Cond})) is satisfied and that $L^{\psi}(\Omega,\P)=L^{\psi,\infty}(\Omega,\P)$ (Proposition \ref{prop-weak-Orlicz}). Without loss of generality, we may assume that $\psi \in \CC ^1(\R)$. According to \cite[Lemma 3.19]{lefevre_compact_2009}, there exists $B>1$ such that
$$\int _1 ^{\infty}\frac{\psi'(u)}{\psi(Bu)}du=\int _{\psi(1)}^{\infty}\frac{1}{\psi(B\psi^{-1}(x))}dx<+\infty.$$
For $C>0$, we denote $\chi_{C,\alpha}(x)=\psi(C\psi^{-1}(x^{N+\alpha+1}))\in \CC^1(\R)$. Using (\ref{eq|eq_order-bound-Cond}) (resp. (\ref{eq|eq_order-bound-Mpsi-Cond}), there exists $A>0$ such that (resp. for every $A>0$)
\begin{eqnarray*}\int _{\B_N}\chi_{C,\alpha}\left(\frac{1}{1-|\phi|}\right)dv_{\alpha} & = & \chi_{C,\alpha}(1) + \int _1^{\infty}v_{\alpha}\left(1-|\phi|<1/t\right)\chi'_{C,\alpha}(t)dt\\
& \leq & \chi_{C,\alpha}(1) + K\int_1^{\infty}\frac{\chi'_{C,\alpha}(t)}{\chi_{A,\alpha}(t)}dt,
\end{eqnarray*}
for some $K<\infty$. Setting $C=A/B$, the change of variable $u=\chi_{C,\alpha}(t)$ gives $\chi_{A,\alpha}(t)=\chi_{B,\alpha}(u)$ and
$$\int _{\B_N}\chi_{C,\alpha}\left(\frac{1}{1-|\phi|}\right)dv_{\alpha}\leq \chi_{C,\alpha}(1)+K\int_{\chi_{C,\alpha}(1)}^{\infty}\frac{du}{\chi_{B,\alpha}(u)}<+\infty.$$
Thus (\ref{eq|cond-simple-OB-Lpsi}) (resp. \ref{eq|cond-simple-OB-Mpsi}) is satisfied and we are done.
\end{proof}

We are now ready to characterize \emph{small} Bergman-Orlicz spaces.

\subsection{A more complete characterization of \emph{small} Bergman-Orlicz spaces}

We recall that the function $w^{\psi}$ is defined in \eqref{weightpsi}. In order to prove the main result of this paragraph, we need the following lemma.

\begin{lemma}\label{lem-main-thm-carac}Let $M\geq 1$, $\alpha \geq -1$, and $\psi$ be an Orlicz function. Set $\gamma = N+\alpha +1$. The following assertions are equivalent.
\begin{enumerate}
\item $A_{\alpha}^{\psi}\left(\mathbb{B}_{M}\right)=H_{w^{\psi,\gamma}}^{\infty}\left(\mathbb{B}_{M}\right)$ with equivalent norms;
\item There exist finitely many functions $f_1,\ldots ,f_d$ in $A_{\alpha}^{\psi}(\B_M)$ such that, for any $f$ in the unit ball of $A_{\alpha}^{\psi}(\B_M)$,
$$|f(z)|\leq |f_1(z)|+\ldots +|f_d(z)|,$$
for any $z\in\B_M$.
\end{enumerate}
\end{lemma}

\begin{proof}Let us assume that (1) holds. We recall first that if $\psi$ is an Orlicz function, then $w^{\psi,\gamma}$ is an essential weight on $\B_N$ by Lemma \ref{lem|weight-essential-ball-spec-case}, so it is equivalent to a log-convex weight by \cite{bierstedt_associated_1998} (see also \cite[Lemma 2.1]{hyvarinen_essential_2012}). By \cite[Theorem 1.3]{abakumov_moduli_2015} (see also \cite[Corollary 12]{abakumov_holomorphic_nodate}), (2) holds for $H_{w^{\psi,\gamma}}^{\infty}\left(\mathbb{B}_{M}\right)$ instead of $A_{\alpha}^{\psi}\left(\mathbb{B}_{M}\right)$, hence the first implication.
Conversely, if (2) holds it is clear from Proposition \ref{prop|point-evaluation-bounded-prop} that any function in $H_{w^{\psi,\gamma}}^{\infty}(\B_N)$ also belongs to $A_{\alpha}^{\psi}(\B_N)$.
\end{proof}

We introduce the symbol
\begin{equation}\label{eq-symbol}
\varphi(z)=\left(\pi(z),0'\right)
\end{equation}
where $0'$ is the null $(N-1)$-tuple and $\pi(z)=N^{N/2}z_{1}z_{2}\ldots z_{N}$
for $z=\left(z_{1},z_{2},\ldots,z_{N}\right)\in\mathbb{B}_{N}$. Note that
$\pi$ maps $\mathbb{B}_{N}$ onto $\mathbb{D}$ and $\overline{\mathbb{B}_{N}}$
onto $\overline{\mathbb{D}}$ by the Arithmetic Mean Inequality. Our first main result states as follows.

\begin{theorem}
\label{thm|thm-equiv-OB-B-Delta2}Let $M\geq 1$, $N>1$, $\alpha>-1$,
and let $\psi$ be an Orlicz function. Set $\gamma = N+\alpha +1$. The following assertions are equivalent.
\begin{enumerate}
\item $\psi$ belongs to the $\Delta^{2}$--class;
\item $A_{\alpha}^{\psi}\left(\mathbb{B}_{M}\right)=H_{w^{\psi,\gamma}}^{\infty}\left(\mathbb{B}_{M}\right)$ with equivalent norms;
\item $A_{\alpha}^{\psi}\left(\mathbb{B}_{M}\right)=H_{w^{\psi}}^{\infty}\left(\mathbb{B}_{M}\right)$ with equivalent norms;
\item There exist finitely many functions $f_1,\ldots ,f_d$ in $A_{\alpha}^{\psi}(\B_M)$ such that, for any $f$ in the unit ball of $A_{\alpha}^{\psi}(\B_M)$,
$$|f(z)|\leq |f_1(z)|+\ldots +|f_d(z)|,\,z\in \B_M;$$
\item For every $\phi:\mathbb{B}_{M}\rightarrow\mathbb{B}_{M}$ holomorphic,
$C_{\phi}$ acting on $A_{\alpha}^{\psi}\left(\mathbb{B}_{M}\right)$
is order-bounded into $L_{\alpha}^{\psi}$;
\item For every $\phi:\mathbb{B}_{N}\rightarrow\mathbb{B}_{N}$ holomorphic,
$C_{\phi}$ is bounded on $A_{\alpha}^{\psi}\left(\mathbb{B}_{N}\right)$;
\item For every $\phi:\mathbb{B}_{N}\rightarrow\mathbb{B}_{N}$ holomorphic,
$C_{\phi}$ is bounded on $AM_{\alpha}^{\psi}\left(\mathbb{B}_{N}\right)$;
\item $C_{\varphi}$ with symbol $\varphi$ defined in (\ref{eq-symbol}) is bounded on $A_{\alpha}^{\psi}\left(\mathbb{B}_{N}\right)$.
\end{enumerate}
\end{theorem}
\begin{proof}Let $M$ be fixed as in the statement of the theorem. We will prove that (1) $\Leftrightarrow$ (2) $\Leftrightarrow$ (3) $\Leftrightarrow$ (4) $\Leftrightarrow$ (5) and that (1) $\Leftrightarrow$ (6) $\Leftrightarrow$ (7). Since (1) does not depend on $M$ nor $N$, this will be enough. (1) $\Rightarrow$ (2) was proven in Paragraph \ref{first-carac-parag}. (2) $\Leftrightarrow$ (4) is Lemma \ref{lem-main-thm-carac}. If (4) holds, since it is equivalent to (2), then Proposition \ref{prop|prop-Cphi-bounded-growth-space} ensures that every composition operator is bounded from $A_{\alpha}^{\psi}\left(\mathbb{B}_{M}\right)$ into itself (note by passing that it gives (6) for the value $M$). Moreover, (4) trivially implies that the canonical embedding from $A_{\alpha}^{\psi}\left(\mathbb{B}_{M}\right)$ into $L_{\alpha}^{\psi}$ is order-bounded into $L_{\alpha}^{\psi}$. Thus, Lemma \ref{lem-OB-gene} gives (5). We divide the proof of (5) $\Rightarrow$ (1) into two parts. First, we show that whatever $N>1$, (6) always imply (1). Since (5) $\Rightarrow$ (6) whenever $M=N$, it will give (5) $\Rightarrow$ (1) in the case $M>1$. Because (6) trivially implies (8), we only have to prove (8) $\Rightarrow$ (1). We recall that $\varphi$ is defined in (\ref{eq-symbol}) by $\varphi(z)= \left(\pi(z),0'\right)$, where $\pi(z)=N^{N/2}z_{1}z_{2}\ldots z_{N}$. Now, as
computed in \cite[Pages 904-905]{maccluer_angular_1986} (by the means of a result due to Ahern \cite[Theorem 1]{Ahern1983}),
\begin{equation}\label{ahern}
\mu_{\varphi,\alpha}\left(S\left(e_{1},h\right)\right)\sim h^{\frac{2\alpha+N+3}{2}}.
\end{equation}
Hence, according to Theorem \ref{thm|carac-cont-embedding-H-O}, $C_{\phi_{\beta}}$
is unbounded on $A_{\alpha}^{\psi}\left(\mathbb{B}_{N}\right)$ whenever
for any $A>0$, there exists $\left(h_{n}\right)_{n}$ decreasing
to $0$ such that
\[
\frac{1}{\psi\left(A\psi^{-1}\left(1/h_{n}^{N+\alpha+1}\right)\right)}=o\left(h_{n}^{\frac{2\alpha+N+3}{2}}\right)\mbox{ as }n\rightarrow\infty.
\]
Using that $\psi^{-1}$ is increasing and setting $y_{n}=A\psi^{-1}\left(1/h_{n}^{N+\alpha+1}\right)$,
this condition is realized as soon as for every $A>0$, every $C>0$,
and every $y_{0}>0$ there exists $y\geq y_{0}$ such that
\[
\psi\left(y\right)^{\frac{2N+2\alpha+2}{2\alpha+N+3}}\geq C\psi\left(\frac{y}{A}\right).
\]
Now, since $N>1$, we have $\frac{2N+2\alpha+2}{2\alpha+N+3}>1$. By Proposition \ref{prop|prop-Orlicz-growth}, the last
condition is satisfied if $\psi$ does not belong to the $\Delta^{2}$--class,
which proves (8) $\Rightarrow$ (1) for any value of $N>1$.

We now turn to the proof of (5) $\Rightarrow$ (1) in the case $M=1$. It is similar to the previous one but even simpler: it is enough to take $\phi(z)=z$, $z\in \mathbb{D}$. Indeed, in this case, $\mu_{\phi}(C_{\alpha}(h))=1-v_{\alpha}(D(0,1-h))\sim h^{1+\alpha}$ as $h$ tends to $0$. Now, since $\frac{2+\alpha}{1+\alpha} >1$, we can prove as above that, if $\psi$ does not satisfy the $\Delta^2$--condition, then for any $A>0$, there exists $(h_n)_n$ decreasing to $0$ such that
\[
\frac{1}{\psi\left(A\psi^{-1}\left(1/h_{n}^{2+\alpha}\right)\right)}=o\left(h_{n}^{1+\alpha}\right)\mbox{ as }n\rightarrow\infty.
\]
This means that $C_{\phi}$ is not order bounded on $A_{\alpha}^{\psi}(\B_N)$ by (1)--(a) of Theorem \ref{thm|thm_2nd_applic}. Thus (5) $\Rightarrow$ (1) also holds in the case $M=1$.

At this point, we have proven  (1) $\Leftrightarrow$ (2) $\Leftrightarrow$ (4) $\Leftrightarrow$ (5) . To get that (3) is also equivalent to the previous assertions, it is enough to recall that (1) $\Rightarrow$ (3) was also shown in Paragraph \ref{first-carac-parag}, and to observe that (3) $\Rightarrow$ (4) can be obtained as (2) $\Rightarrow$ (4) in Lemma \ref{lem-main-thm-carac}.

To finish, we remark that (6) $\Leftrightarrow$ (7) is contained in Theorem \ref{thm|weak-star-topo-H-O} (4), that (1) $\Rightarrow$ (6) is given by Corollary \ref{coro|Cphi-bounded-BO-Delta2}, that (6) $\Rightarrow$ (8) is trivial, and that (8) $\Rightarrow$ (1) has been proven just above. This concludes the proof.
\end{proof}
\begin{unremark}1) Note that the symbol $\varphi$ appearing in the eighth assertion of Theorem \ref{thm|thm-equiv-OB-B-Delta2} is the simplest and most classical example of symbol inducing an unbounded composition operator on $H^2(\B_N)$ and $A_{\alpha}^2(\B_N)$. One of the striking aspect of Theorem \ref{thm|thm-equiv-OB-B-Delta2} is that the boundedness of this composition operator on $A_{\alpha}^{\psi}\left(\mathbb{B}_{M}\right)$ entirely determines the \emph{nature} of this space, and the boundedness of \emph{all} composition operators on it.

2) In the previous theorem, the eight assertions are also equivalent to the following nineth one:
\begin{enumerate}\setcounter{enumi}{8}\item There exists an essential (radial) weight $v$ such that $A_{\alpha}^{\psi}\left(\mathbb{B}_{M}\right)=H_{v}^{\infty}\left(\mathbb{B}_{M}\right)$ with equivalent norms.
\end{enumerate}
Indeed, the proof of (9) $\Rightarrow$ (4) works the same as that of (2) $\Rightarrow$ (4) (or equivalently (1) $\Rightarrow$ (2) in Lemma \ref{lem-main-thm-carac}).

3) As already said in Remark \ref{rem|rem-HO-growth}, Assertion (2) in Theorem \ref{thm|thm-equiv-OB-B-Delta2} is never true for Hardy-Orlicz spaces, so the previous theorem does not hold for $\alpha =-1$. In fact, Assertions (2), (3), (4) and (5) do not hold for $H^{\psi}\left(\mathbb{B}_{N}\right)$. Nevertheless we know after \cite[Theorem 3.7]{charpentier_composition_2011}
that every composition operator is bounded on $H^{\psi}\left(\mathbb{B}_{N}\right)$
whenever $\psi$ satisfies the $\Delta^{2}$--condition. Moreover we can
prove as in Theorem \ref{thm|thm-equiv-OB-B-Delta2} that every composition operator is bounded on $H^{\psi}\left(\mathbb{B}_{N}\right)$ iff $\psi$ satisfying the $\Delta^{2}$--condition, and iff the single composition operator $C_{\varphi}$ is bounded on $H^{\psi}\left(\mathbb{B}_{N}\right)$, where $\varphi$ is given by (\ref{eq-symbol}) (the estimate (\ref{ahern}) also holds for $\alpha =-1$). Actually, these properties are also equivalent to the boundedness on $H^{\psi}\left(\mathbb{B}_{N}\right)$ of the single composition operator induced by the symbol $\widetilde{\varphi}(z)=\left(\theta(z),0'\right)$,
where $\theta$ is a non-constant inner function of $\mathbb{B}_{N}$
vanishing at the origin \cite{alexandrov_existence_1983}. Since $\theta$ is measure-preserving
as a map from $\partial\mathbb{B}_{N}$ to $\partial\mathbb{D}$,
we get
\[
\sigma_{N}\widetilde{\varphi}^{-1}\left(\mathcal{S}\left(e_{1},h\right)\right)=\sigma_{N}\left(\theta^{-1}\left(\mathcal{S}\left(1,h\right)\right)\right)=h.
\]
Thus $C_{\widetilde{\varphi}}$ is unbounded on $H^{\psi}\left(\mathbb{B}_{N}\right)$
whenever for every $A>0$, there exists $\left(h_{n}\right)_{n}$
decreasing to $0$ such that
\[
\frac{1}{\psi\left(A\psi^{-1}\left(1/h_{n}^{N}\right)\right)}=o\left(h_{n}\right)\mbox{ as }n\rightarrow\infty
\]
which is, as above, satisfied whenever $\psi$ does not belong to
the $\Delta^{2}$--class, keeping in mind that $N>1$.
\end{unremark}

\subsection{Compactness and order boundedness}

\subsubsection{Two direct consequences of Theorem \ref{thm|thm-equiv-OB-B-Delta2}}

In this paragraph we derive from Theorem \ref{thm|thm-equiv-OB-B-Delta2} and from known results some new statements (not necessarily new results) about the compactness of composition operators on small weighted Banach spaces and small Bergman-Orlicz spaces.

First of all, up to now, no estimate of the essential norm $\left\Vert C_{\phi} \right\Vert _e$ of the composition operator $C_{\phi}$ on Bergman-Orlicz spaces is known, except when $\psi(x)=x^p$, $1\leq p\leq \infty$ (see \cite{charpentier_essential_2012} and the references therein). In \cite{bonet_essential_1999} the authors estimate $\left\Vert C_{\phi} \right\Vert _e$ for composition operators acting on weighted Banach spaces on certain domains in the complex plane, including the unit disc. This, together with Lemma \ref{lem|weight-essential-ball-spec-case} and Theorem \ref{thm|thm-equiv-OB-B-Delta2}, provides us with the estimates in Equation \eqref{norm-ess} below. We omit the proof.

\begin{corollary}Let $\alpha > -1$, let $\psi$ be an Orlicz function satisfying the $\Delta^2$--condition, and let $\phi$ be a holomorphic self-map of $\D$. Then
\begin{equation}\label{norm-ess}\left\Vert C_{\phi} \right\Vert _{e,A_{\alpha}^{\psi}} \simeq \lim _{r\rightarrow 1}\sup_{|\phi(z)|>r}\frac{\psi^{-1}\left(1/\left(1-\left|\phi\left(z\right)\right|\right)\right)}{\psi^{-1}\left(1/\left(1-\left|z\right|\right)\right)} \simeq \limsup _{n\rightarrow \infty}\frac{\left\Vert \phi^n \right\Vert _{A_{\alpha}^{\psi}}}{\left\Vert z^n \right\Vert _{A_{\alpha}^{\psi}}}.
\end{equation}
In particular $C_{\phi}$ is compact on $A_{\alpha}^{\psi}(\D)$ iff
\begin{equation}\label{ess1}\limsup_{|z|\rightarrow 1}\frac{\psi^{-1}\left(1/\left(1-\left|\phi\left(z\right)\right|\right)\right)}{\psi^{-1}\left(1/\left(1-\left|z\right|\right)\right)}=0
\end{equation}
or iff
\begin{equation}\label{ess2}\limsup _{n\rightarrow \infty}\frac{\left\Vert \phi^n \right\Vert _{A_{\alpha}^{\psi}}}{\left\Vert z^n \right\Vert _{A_{\alpha}^{\psi}}}=0.
\end{equation}
\end{corollary}

Similar estimates also appear in \cite{hyvarinen_essential_2012,montes-rodriguez_weighted_2000} where weighted composition operators acting on weighted Banach spaces of the unit disc are considered. Note that the equivalence between (\ref{ess1}) and the compactness of $C_{\phi}$ also holds in the unit ball $\B_N$ instead of $\D$, as a particular case of \cite[Theorem 3.11]{charpentier_compact_2013}. Indeed, it is shown there that, under some mild condition, this characterization holds not only for $\B_N$, but also on the whole range of weighted Begman-Orlicz spaces, whenever the Orlicz functions satisfy the $\nabla _0$--condition (yet the essential norm was not estimated, even in $\D$). Moreover, we believe that some part of the proof given in \cite{bonet_essential_1999} smoothly extends to $\B_N$, so that it is quite likely that the first estimate in (\ref{norm-ess}) is also true in the unit ball.

\medskip{}

Second it was proven in \cite[Theorem 3.1]{charpentier_compact_2013} that if $C_{\phi}$ is compact on $A_{\alpha}^{\psi}(\B_N)$ for every Orlicz function $\psi$, then $C_{\phi}$ is compact on $H^{\infty}$. Actually it is not difficult to adapt the proof of that theorem (more specifically that of Lemma 3.2 there) to show that if $C_{\phi}$ is compact on $A_{\alpha}^{\psi}(\B_N)$ for every Orlicz function $\psi$ \emph{satisfying the $\Delta^2$--condition}, then $C_{\phi}$ is compact on $H^{\infty}$. We deduce from that and Theorem \ref{thm|thm-equiv-OB-B-Delta2} the following.

\begin{corollary}Let $\phi$ be a holomorphic self-map of $\B_N$. If $C_{\phi}$ is compact on $H_v^{\infty}(\B_N)$ for every weight $v$ of the form $v=w^{\psi}$ with $\psi$ an Orlicz function satisfying the $\Delta^2$--condition, then $C_{\phi}$ is compact on $H^{\infty}$.
\end{corollary}

This corollary is in fact a slight refinement of \cite[Corollary 3.8]{bonet_composition_1998}.

\subsubsection{Compactness and order boundedness}


The main result of this paragraph states as follows:
\begin{theorem}\label{main-thm-compacntess-Deltaup2}
Let $\alpha\geq-1$, let $\psi$ be an Orlicz function satisfying
the $\Delta^{2}$--condition and let $\phi:\mathbb{B}_{N}\rightarrow\mathbb{B}_{N}$
be holomorphic. The following assertions are equivalent.
\begin{enumerate}
\item $C_{\phi}$ is compact on $A_{\alpha}^{\psi}\left(\mathbb{B}_{N}\right)$;
\item $C_{\phi}$ acting on $A_{\alpha}^{\psi}\left(\mathbb{B}_{N}\right)$
is order bounded into $M_{\alpha}^{\psi}$;
\item $C_{\phi}$ is weakly compact on $A_{\alpha}^{\psi}\left(\mathbb{B_N}\right)$.
\end{enumerate}
\end{theorem}
\begin{unremark}\label{remark-OB-comp-Delta^1}(1) The two first equivalences were obtained in \cite[Theorem 3.24]{lefevre_compact_2009} for $N=1$ and $\alpha =-1$.\\
(2) The $\Delta^2$--Condition in Theorem \ref{main-thm-compacntess-Deltaup2} is certainly important for (1)$\Leftrightarrow$(2). Indeed it was proven in \cite[Theorem 4.22]{lefevre_compact_2009} that there exists an Orlicz function $\psi\in \Delta ^1$, and an analytic self-map $\phi$ of $\D$ such that the composition operator $C_{\phi}:H^{\psi}(\D)\rightarrow H^{\psi}(\D)$ is not order bounded into $M^{\psi}(\partial \D)$, though it is compact.\\
(3) Yet the $\Delta ^2$--condition has no importance for (2)$\Leftrightarrow$(3) which in fact holds for a larger class of Orlicz functions, namely the $\Delta ^0$--class, see \cite[Theorem 4.21]{lefevre_compact_2009}. But this equivalence fails if $\psi\in \Delta _2\cap \nabla _2$ because, in this case, $A_{\alpha}^{\psi}(\B_N)$ is reflexive and so every bounded operator in weakly compact.\\
(4) The equivalence between compactness and weak compactness for a composition operator is more generally related to the duality properties of the spaces on which it acts. For instance, such equivalence was also observed in the context of weighted Banach spaces of analytic functions on the unit disc \cite{bonet_essential_1999}.
\end{unremark}
\begin{proof}[Proof of Theorem \ref{main-thm-compacntess-Deltaup2}]We present it only for $\alpha >-1$, as it is again very similar if $\alpha =-1$. The proof of (1) $\Rightarrow$ (2) $\Rightarrow$ (3) is an adaptation of the proof of Theorem 3.24 in \cite{lefevre_compact_2009}.  By Proposition \ref{OBMpsiimpliescomp}, we already know (2) $\Rightarrow$ (1). To prove the converse, it is enough to show, according to Theorem \ref{thm|thm_2nd_applic},
that the condition
\begin{equation}
\left\Vert f_{a}\circ\phi\right\Vert _{A_{\alpha}^{\psi}\left(\mathbb{B}_{N}\right)}=o_{\left|a\right|\rightarrow1}\left(\frac{1}{\psi^{-1}\left(1/\left(1-\left|a\right|\right)^{N+\alpha+1}\right)}\right)\label{eq|eq_last_thm_O-B}
\end{equation}
implies
\[
\mu_{\phi}\left(C_{\alpha}\left(h\right)\right)\leq\frac{C_{A}}{\psi\left(A\psi^{-1}\left(1/h^{N+\alpha+1}\right)\right)}
\]
for every $A>0$, every $h\in\left(0,h_{A}\right)$ and some $h_{A}\in\left(0,1\right)$ and $C_A>0$.
(note that we may just require the preceding to be true for any $A$ bigger than some arbitrary positive constant). Now let us remind that we showed, at the end of the proof of Corollary \ref{cor|corollaire_Berezin_cont_comp}, that if Condition \eqref{eq|eq_last_thm_O-B} holds then for every $\varepsilon > 0$, there exists $h_{\varepsilon}\in (0,1)$ such that for every $\zeta \in \mathbb{S}_N$ and every $h\in (0,h_{\varepsilon})$ we have
\begin{equation}\label{todaytricebis}\mu_{\phi,\alpha}(S(\zeta, h))\leq \frac{1}{\psi\left(\frac{K}{\varepsilon}\psi^{-1}\left(1/h^{N+\alpha+1}\right)\right)}.\end{equation}
Now it is a classical geometric fact that there exists a constant $M>0$, depending only on $\alpha$ such that for any $h\in\left(0,1\right)$, there exists $\zeta\in\mathbb{S}_{N}$ such that
\[
\mu\left(C_{\alpha}\left(h\right)\right)\leq\frac{M}{h^{N+\alpha+1}}\mu\left(S\left(\zeta,Kh\right)\right)\text{ \cite{rudin_function_1980}}
\]
Thus we get that for every $\varepsilon > 0$, there exists $h_{\varepsilon}\in (0,1)$ such that for every $h\in (0,h_{\varepsilon})$
$$\mu\left(C_{\alpha}\left(h\right)\right)\leq \frac{M}{h^{N+\alpha+1}\psi\left(\frac{K}{\varepsilon}\psi^{-1}\left(1/h^{N+\alpha+1}\right)\right)}.$$
Since $\psi$ satisfies the $\Delta^2$--condition, there exists a constant $C>0$ such that for every $x$ large enough, $(\psi(x))^2\leq \psi(Cx)$. Then for any $\varepsilon>0$ small such that $K/(C\varepsilon)>1$, there exists $h'_{\varepsilon}$ such that for every $h\in (0,h'_{\varepsilon})$, we have
$$\frac{M}{h^{N+\alpha+1}\psi\left(\frac{K}{\varepsilon}\psi^{-1}\left(1/h^{N+\alpha+1}\right)\right)} \leq \frac{M\psi\left(\frac{K}{C\varepsilon}\psi^{-1}(1/h^{N+\alpha+1})\right)}{\left[\psi\left(\frac{K}{C\varepsilon}\psi^{-1}\left(1/h^{N+\alpha+1}\right)\right)\right]^2} = \frac{M}{\psi\left(\frac{K}{C\varepsilon}\psi^{-1}\left(1/h^{N+\alpha+1}\right)\right)}.$$
We conclude by setting $A=K/(C\varepsilon)$ and choosing $h_A=\min(h_{\varepsilon},h'_{\varepsilon})$.

\smallskip{}

To prove (1) $\Leftrightarrow$ (3), we need only to prove (3) $\Rightarrow$ (1) and then, according to Corollary \ref{cor|corollaire_Berezin_cont_comp}, that the weak compactness of $C_{\phi}$ on $A_{\alpha}^{\psi}\left(\mathbb{B}_N\right)$ implies
\begin{equation}\label{eqthmweak}\left\Vert f_{a}\circ\phi\right\Vert _{A_{\alpha}^{\psi}\left(\mathbb{B}_{N}\right)}=o_{\left|a\right|\rightarrow1}\left(\frac{1}{\psi^{-1}\left(1/\left(1-\left|a\right|\right)^{N+\alpha+1}\right)}\right).
\end{equation}
But this can be done exactly in the same way as in the proof of \cite[Theorem 3.20]{lefevre_compact_2009}. Indeed, by Theorem \ref{thm|weak-star-topo-H-O} (4), $C_{\phi}:AM_{\alpha}^{\psi}(\B_N)\rightarrow AM_{\alpha}^{\psi}(\B_N)$ is also weakly compact so that we can appeal to \cite[Theorem 4]{lefevre_criterion_2008} and get, for any $\varepsilon >0$, the existence of a constant $K_{\varepsilon}>0$ such that for every $f\in AM_{\alpha}^{\psi}\left(\mathbb{B}_N\right)$,
$$\left\Vert f\circ \phi\right\Vert_{A_{\alpha}^{\psi}}\leq K_{\varepsilon}\left\Vert f \right\Vert _{L_{\alpha}^1}+\varepsilon \left\Vert f \right\Vert_{A_{\alpha}^{\psi}}.$$
Taking $f=f_a$, we obtain
$$\left\Vert f_a\circ \phi\right\Vert_{A_{\alpha}^{\psi}}\leq K_{\varepsilon}(1-|a|)^{N+\alpha+1}+\frac{2^{N+\alpha+1}\varepsilon}{\psi^{-1}\left(1/(1-|a|)^{N+\alpha+1}\right)},$$
since $\left\Vert f_a \right\Vert _{L_{\alpha}^1}=(1-|a|)^{N+\alpha+1}$ (a change a variable for $\alpha > -1$ and the Cauchy formula when $\alpha = -1$). We conclude that (\ref{eqthmweak}) holds, for $\psi(x)/x \rightarrow +\infty$ as $x\rightarrow +\infty$.
\end{proof}

\section{Some geometric conditions for the order boundedness of composition
operators on Hardy-Orlicz and Bergman-Orlicz spaces}

For $\zeta\in\mathbb{S}_{N}$
and $a>1$, we recall that the Kor\'anyi approach region $\Gamma\left(\zeta,a\right)$
of angular opening $a$ is defined by
\[
\Gamma\left(\zeta,a\right)=\left\{ z\in\mathbb{B}_{N},\,\left|1-\left\langle z,\zeta\right\rangle \right|<\frac{a}{2}\left(1-\left|z\right|^{2}\right)\right\} .
\]
When $N=1$, Kor\'anyi approach regions reduce to non-tangential approach
regions (or Stolz domains) given by
\[
\Gamma\left(\zeta,a\right)=\left\{ z\in\mathbb{D},\,\left|\zeta-z\right|<\frac{a}{2}\left(1-\left|z\right|^{2}\right)\right\} .
\]

In \cite{MacCluer_95}, the author proves that if $\phi$ takes $\mathbb{B}_{N}$
into a Kor\'anyi approach region, then $C_{\phi}$ is bounded or
compact on $H^{p}\left(\mathbb{B}_{N}\right)$ (hence on $A_{\alpha}^{p}\left(\mathbb{B}_{N}\right)$
for any $\alpha\geq-1$) whenever the angular opening of the Kor\'anyi
region is not too large. This result immediately extends to the $H^{\psi}\left(\mathbb{B}_{N}\right)$
or $A_{\alpha}^{\psi}\left(\mathbb{B}_{N}\right)$ setting for Orlicz
functions $\psi$ satisfying the $\Delta_{2}$--condition \cite[Theorem 3.3]{charpentier_compact_2013}.

On the opposite side we have the following.

\begin{Thmuncount}[Theorem 3.5 of \cite{charpentier_compact_2013}]\label{thm|thm_pas_comp_delta_2_Kor_app}Let
$N\geq1$ and let $\psi$ be an Orlicz function satisfying the $\Delta^{2}$--condition.
Then, for every $\zeta\in\mathbb{S}_{N}$ and every $b>1$, there
exists a holomorphic self-map $\phi$ taking $\mathbb{B}_{N}$ into
$\Gamma\left(\zeta,b\right)$, such that $C_{\phi}$ is not compact
on $H^{\psi}\left(\mathbb{B}_{N}\right)$ (and hence not order-bounded
into $M^{\psi}\left(\mathbb{S}_{N},\sigma _N\right)$). The same holds for
the weighted Bergman-Orlicz spaces in dimension $1$.\end{Thmuncount}

So, when $\psi \in \Delta^2$, the condition $\phi\left(\mathbb{B}_{N}\right)\subset\Gamma\left(\zeta,b\right)$, whatever the opening $b>1$,  becomes non-sufficient for the compactness of $C_{\phi}$ on $A_{\alpha}^{\psi}$, while it is trivially sufficient for its order boundedness into $L_{\alpha}^{\psi}$ (since in this case every composition operator is order bounded into $L_{\alpha}^{\psi}$ by Theorem \ref{thm|thm-equiv-OB-B-Delta2}).

\smallskip{}

The next result tells that if $\phi(\mathbb{B}_{N}) \subset \Gamma\left(\zeta,a\right)$ for some $\zeta\in\mathbb{S}_{N}$ and $a>1$ sufficiently small, then $C_{\phi}:A_{\alpha}^{\psi}(\B_N)\rightarrow A_{\alpha}^{\psi}(\B_N)$ is still order bounded into $L_{\alpha}^{\psi}$ when $\psi$ satisfies the $\Delta^1$--condition.
\begin{theorem}
\label{thm|thm_Kor_order-bounded}Let $\alpha\geq-1$ and let $\psi$
be an Orlicz function.
\begin{enumerate}
\item Let $\phi:\mathbb{D}\rightarrow\mathbb{D}$ be holomorphic. If $\phi\left(\mathbb{D}\right)$
is contained in some Stolz domain, then $C_{\phi}:A_{\alpha}^{\psi}(\D)\rightarrow A_{\alpha}^{\psi}(\D)$
is order bounded into $L_{\alpha}^{\psi}$.
\item Let $N>1$ and let $\phi:\mathbb{B}_{N}\rightarrow\mathbb{B}_{N}$
be holomorphic. We set $a_{N}=1/\cos\left(\frac{\pi}{2}\frac{\alpha+2}{N+\alpha+1}\right)$.
\begin{enumerate}
\item If $\phi\left(\mathbb{B}_{N}\right)\subset\Gamma\left(\zeta,\gamma\right)$
for $\gamma<a_{N}$ then $C_{\phi}:A_{\alpha}^{\psi}(\B_N)\rightarrow A_{\alpha}^{\psi}(\B_N)$ is order bounded into $L_{\alpha}^{\psi}$;
\item We assume that $\psi$ satisfies the $\Delta^{1}$--condition. If $\phi\left(\mathbb{B}_{N}\right)\subset\Gamma\left(\zeta,a_{N}\right)$,
then $C_{\phi}:A_{\alpha}^{\psi}(\B_N)\rightarrow A_{\alpha}^{\psi}(\B_N)$ is order bounded into $L_{\alpha}^{\psi}$.
\end{enumerate}
%
\end{enumerate}
\end{theorem}
\begin{proof}
We will assume that $\alpha>-1$, the case $\alpha=-1$ being similarly proven (to get the case
$\alpha=-1$, it is essentially enough to change suitably $\mathbb{B}_{N}$
into $\mathbb{S}_{N}$ and $v_{\alpha}$ into $\sigma_{N}$, and to
consider $\phi^{*}$ instead of $\phi$, or $\mathcal{S}(\zeta,h)$
instead of $S(\zeta,h)$. The details are left to the reader). Moreover
we can assume for simplicity that all Kor\'anyi approach regions
$\Gamma\left(\zeta,a\right)$ appearing in the sequel are based at
the point $\zeta=e_{1}=\left(1,0,\ldots,0\right)\in\mathbb{S}_{N}$.
Indeed for every unitary maps $U$ of $\mathbb{B}_{N}$ $U\Gamma\left(\zeta,a\right)=\Gamma\left(U\zeta,a\right)$,
and $C_{\phi}:A_{\alpha}^{\psi}(\B_N)\rightarrow A_{\alpha}^{\psi}(\B_N)$ is order bounded into $L_{\alpha}^{\psi}$
if and only if $C_{U\phi}:A_{\alpha}^{\psi}(\B_N)\rightarrow A_{\alpha}^{\psi}(\B_N)$ is order bounded into $L_{\alpha}^{\psi}$
(by (1) in Proposition \ref{prop|prop-carac-OB-evaluation}).

The proof of (1) and (2)--(a) are similar: if $N=1$, we recall that if $\phi$ maps $\D$ into a Stolz domain, then $C_{\phi}$ is Hilbert-Schmidt on $A_{\alpha}^2(\D)$ \cite{shapiro_composition_1993} and then the conclusion follows by Remark \ref{today} (1).
If $N>1$ it is classical that if $\phi\left(\mathbb{B}_{N}\right)\subset\Gamma\left(\zeta,\gamma\right)$
for $\gamma<a_{N}$ then $C_{\phi}$ is Hilbert-Schmidt on $A_{\alpha}^{2}\left(\mathbb{B}_{N}\right)$ (see the remarks on Theorem 2.2 in \cite{maccluer_compact_1985}, Page 246). We conclude again by Remark \ref{today}.

We now deal with (2)--(b). For this purpose, we need first an adaptation of \cite[Lemma 6.3]{cowen_composition_1995} to the Orlicz setting.
\begin{lemma}
Let $\psi$ be an Orlicz function and $\alpha\geq-1$. Let $\varphi:\mathbb{D}\rightarrow\mathbb{D}$
be holomorphic and assume that $\varphi\left(\mathbb{D}\right)\subset\Gamma\left(1,\gamma\right)$
for some $\gamma>1$. There exists a constant $A>0$ which depends
only on $\phi\left(0\right)$ and $\gamma$ such that
\[
\mu_{\varphi,\alpha}\left(S\left(1,h\right)\right)\leq\frac{1}{\psi\left(A\psi^{-1}\left(1/h^{\beta(\alpha+2)}\right)\right)},
\]
where ${\displaystyle \beta=\frac{\pi}{2\cos^{-1}\left(1/\gamma\right)}}$.\end{lemma}
\begin{proof}
Since $\varphi\left(\mathbb{D}\right)\subset\Gamma\left(1,\gamma\right)$,
the beginning of the proof of \cite[Lemma 6.3]{cowen_composition_1995} ensures
that ${\displaystyle \frac{1}{1-\varphi\left(z\right)}=F\circ\tau}$
where ${\displaystyle F\left(z\right)=\left(\frac{1+z}{1-z}\right)^{2b/\pi}}$
with $b=\cos^{-1}\left(1/\gamma\right)$ and $\tau$ is an analytic self-map
of $\mathbb{D}$. It follows that
\begin{eqnarray*}
\left\{ z\in\mathbb{D},\thinspace\left|1-\varphi\left(z\right)\right|<h\right\}  & = & \left\{ z\in\mathbb{D},\thinspace\left|F\circ\tau\left(z\right)\right|>1/h\right\} \\
 & \subseteq & \left\{ z\in\mathbb{D},\thinspace\left|1-\tau\left(z\right)\right|<2h^{\pi/\left(2b\right)}\right\} .
\end{eqnarray*}
Hence, since every composition operator is bounded on every $A_{\alpha}^{\psi}\left(\mathbb{D}\right)$,
$\alpha\geq-1$, the measure $\mu_{\tau,\alpha}=v_{\alpha}\circ\tau^{-1}$ satisfies, by the Fact \ref{factbounded},
\[
v_{\alpha}\left(\varphi^{-1}\left(S\left(1,h\right)\right)\right)\leq\frac{1}{\psi\left(A\psi^{-1}\left(1/h^{\beta(\alpha+2)}\right)\right)}\text{ (recall that }\beta=\frac{\pi}{2b}\text{)},
\]
for any $h\in\left(0,\eta\right)$, some $\eta\in (0,1)$ and some constant $A>0$ depending
only on $\alpha$ and $\varphi(0)$.
\end{proof}
We finish the proof of (2)--(b) of Theorem \ref{thm|thm_Kor_order-bounded}, assuming
that $\phi\left(\mathbb{B}_{N}\right)\subset\Gamma\left(\zeta,a_{N}\right)$.
The absorption property of the non-isotropic balls $S(\zeta,t)$,
$\zeta\in\mathbb{S}_{N}$, ensures that there exists a constant
$1\leq C<\infty$ such that $C_{\alpha}\left(h\right)\cap\Gamma\left(1,a\right)\subset S\left(1,Ch\right)$.
Thus
\begin{equation}
\mu_{\phi,\alpha}\left(C_{\alpha}\left(h\right)\right)\leq\mu_{\phi,\alpha}\left(S\left(e_{1},Ch\right)\right).\label{eq|eq1-geo-cond-BO}
\end{equation}
Let $\chi_{\phi^{-1}\left(S\left(e_{1},h\right)\right)}$ be the characteristic
function of $\phi^{-1}\left(S\left(e_{1},h\right)\right)$. By integrating
in polar coordinates and applying the slice integration formula we get,
for any $h\in(0,1)$,
\begin{eqnarray}
\mu_{\phi,\alpha}\left(S\left(e_{1},h\right)\right) & = & \int_{\mathbb{B}_{N}}\chi_{\phi^{-1}\left(S\left(e_{1},h\right)\right)}dv_{\alpha}\nonumber \\
 & = & 2N\int_{0}^{1}\int_{\mathbb{S}_{N}}\chi_{\phi^{-1}\left(S\left(e_{1},h\right)\right)}(r\zeta)d\sigma_{N}(\zeta)\left(1-r^{2}\right)^{\alpha}r^{2N-1}dr\nonumber \\
 & = & 2N\int_{0}^{1}\int_{\mathbb{S}_{N}}\int_{0}^{2\pi}\chi_{\phi^{-1}\left(S\left(e_{1},h\right)\right)}\left(e^{i\vartheta}r\zeta\right)\frac{d\vartheta}{2\pi}d\sigma_{N}\left(\zeta\right)\left(1-r^{2}\right)^{\alpha}r^{2N-1}dr\nonumber \\
 & \leq & 2N\int_{0}^{1}\int_{\mathbb{S}_{N}}\int_{0}^{2\pi}\chi_{\phi^{-1}\left(S\left(e_{1},h\right)\right)}\left(e^{i\vartheta}r\zeta\right)\frac{d\vartheta}{2\pi}d\sigma_{N}\left(\zeta\right)\left(1-r^{2}\right)^{\alpha}rdr.\label{eq|eq2-geo-cond-BO}
\end{eqnarray}
Let us now denote by $\phi^{\zeta}:\mathbb{D}\rightarrow\mathbb{D}$
the holomorphic self-map of $\mathbb{D}$ defined by $\phi^{\zeta}\left(z\right)=\phi_{1}\left(z\zeta\right)$
for $z\in\mathbb{D}$. Since $\phi\left(\mathbb{B}_{N}\right)\subset\Gamma\left(e_{1},a_{N}\right)$
then $\phi^{\zeta}\left(\mathbb{D}\right)\subset\Gamma\left(1,a_{N}\right)$
and, for any $\zeta\in\mathbb{S}_{N}$, $\phi(z\zeta)\in S\left(e_{1},h\right)$
if and only if $\phi^{\zeta}\left(z\right)\in S(1,h)$. Moreover $\phi\left(0\right)=\phi^{\zeta}\left(0\right)$
for every $\zeta\in\mathbb{S}_{N}$. Applying the previous lemma to
$\phi^{\zeta}$ and $\gamma=a_{N}$, it follows, for any $\zeta \in \mathbb{S}_N$,
\begin{eqnarray*}
2\int_{0}^{1}\int_{0}^{2\pi}\chi_{\phi^{-1}\left(S\left(e_{1},h\right)\right)}\left(e^{i\vartheta}r\zeta\right)\frac{d\vartheta}{2\pi}\left(1-r^{2}\right)^{\alpha}rdr & = & v_{\alpha}\left(\left(\phi^{\zeta}\right)^{-1}\left(S\left(1,h\right)\right)\right)\\
 & \leq & \frac{1}{\psi\left(A\psi^{-1}\left(1/h^{N+\alpha+1}\right)\right)}.
\end{eqnarray*}
From (\ref{eq|eq1-geo-cond-BO}) and (\ref{eq|eq2-geo-cond-BO}) we
conclude the proof of (2)--(b) using the concavity of $\psi^{-1}$
and Theorem \ref{thm|thm_2nd_applic} (1)--(b), $\psi$ satisfying
the $\Delta^{1}$--condition.\end{proof}
\begin{remark}Point (2)--(b) in the previous theorem is false if $\psi$ satisfies the $\Delta_{2}$--condition. Indeed in this case, the order boundedness of a composition operator implies its compactness and it is known \cite{cowen_composition_1995} (see also \cite{charpentier_compact_2013}) that there exists holomorphic self-map $\phi$ of $\B_N$ such that $\phi\left(\mathbb{B}_{N}\right)\subset\Gamma\left(\zeta,a_{N}\right)$, though $C_{\phi}$ is not compact on $A_{\alpha}^{\psi}\left(\mathbb{B}_{N}\right)$, $\alpha\geq-1$.
\end{remark}

In view of the previous, it is more accurate to consider known geometric conditions to the compactness of a composition operator on Hardy(-Orlicz) or Bergman(-Orlicz) spaces as conditions to its order boundedness. This leads to the following problems:
\begin{enumerate}\item To find reasonable geometric conditions - somehow similar to that involving Kor\'anyi regions but depending on the Orlicz function - sufficient for the compactness of $C_{\phi}$ at every scale of Hardy-Orlicz (or Bergman-Orlicz) spaces.
\item In the classical cases $H^p$ and $A_{\alpha}^p$, to exhibit a geometric condition ensuring the compactness of $C_{\phi}$ but not necessarily its order boundedness. As mentioned in \cite[Page 147--148]{cowen_composition_1995}, it already requires some effort to find an example of a composition operator which is compact but not Hilbert-Schmidt on $H^2(\D)$ (\emph{i.e.} not order bounded), see also \cite[Section 4]{shapiro_compact_1973}.
\item Also it would be interesting to find \emph{where} in the scale of Hardy-Orlicz or Bergman-Orlicz spaces - in terms of the type of growth of Orlicz functions - the order boundedness of $C_{\phi}$ becomes no longer stronger than its compactness.
\end{enumerate}

\section*{Acknowledgements}The author was partly supported by the grant ANR-17-CE40-0021 of the French National Research Agency ANR (project Front).

\end{document}